\documentclass[12pt]{article}
\usepackage{amsmath, amsthm}\usepackage{enumerate}\usepackage{amssymb}\usepackage{bbold}
\usepackage{bbm}\usepackage{dsfont}\usepackage{color}\usepackage{xcolor}\usepackage[explicit]{titlesec}
\newtheorem{theorem}{Theorem}[subsection]
\newtheorem{proposition}[theorem]{Proposition}

\newtheorem{lemma}[theorem]{Lemma}
\newtheorem{example}[theorem]{Example}
\newtheorem{corollary}[theorem]{Corollary}
\newtheorem{remark}[theorem]{Remark}
\newtheorem{assertion}[theorem]{Assertion}

\newtheorem{definition}[theorem]{Definition}

\makeatletter
\renewcommand{\subsection}{\@startsection{subsection}{1}
{0pt}{3.25ex plus 1ex minus.2ex}{-1em}{\normalfont\normalsize\bf}}
\date{}
\makeatother
\begin{document}

\title{{\bf Generalizations of  L- and M-weakly compact operators}}
\maketitle
\author{\centering{{Safak Alpay$^{1}$, Eduard Emelyanov$^{2}$, Svetlana Gorokhova $^{3}$\\ 
\small $1$ Middle East Technical University, Ankara, Turkey\\ 
\small $2$ Sobolev Institute of Mathematics, Novosibirsk, Russia\\ 
\small $3$ Uznyj matematiceskij institut VNC RAN, Vladikavkaz, Russia}

\abstract{L- and M-weakly compact operators were introduced by Meyer-Nieberg in the beginning of seventies in attempts 
of a diversification of the concept of weakly compact operators via imposing Banach lattice structure on the range or on 
the domain of operators. We investigate regularity and algebraic properties of various generalizations 
of L- and M-weakly compact operators from a unified point of view via using regularly P-operators.}

\vspace{5mm}
%$*$Corresponding Author}%%\abstract{.}\\ 


}}
\bigskip
\bigskip

%%%%%%%%%%%%%%%%%%%%
\section{Introduction and Preliminaries}
%%%%%%%%%%%%%%%%%%%%

L- and M-weakly compact operators were introduced half a century ago
in order to diversify the concept of weakly compact operators by adding lattice structure 
on the range/domain of operators.  Since then, these operators attract permanent attention and 
inspire further investigations. Apart of study of already known classes of such operators, we 
introduce and investigate new classes of limitedly \text{\rm LW}- and 
bi-limitedly \text{\rm LW}-operators.

\subsection{}
In the present paper: vector spaces are real; operators are linear; the letters $X$, $Y$, $Z$ 
stand for Banach spaces, $E$, $F$, $G$ for Banach lattices, and $V$, $W$ for ordered vector spaces. 
An operator $T:X\to Y$ (a subset $A$ in $X$) is said to be {\em bounded} 
if $T$ is norm continuous (resp., if $A$ is norm bounded). We denote by $B_X$ the closed unit ball of $X$;
by $\text{\rm L}(X,Y)$ ($\text{\rm K}(X,Y)$, $\text{\rm W}(X,Y)$) the space of all bounded (compact, weakly compact) 
operators from $X$ to $Y$; by $E_+$ the positive cone of $E$; by $\text{\rm sol}(A)$ the solid hull of $A\subseteq E$; 
and by $E^a=\{x\in E: |x|\ge x_n\downarrow 0\Rightarrow\|x_n\|\to 0\}$ the \text{\rm o}-continuous part of $E$. 
An operator $T:V\to W$ is called {\em regular} if $T=T_1-T_2$ for some positive operators $T_1,T_2:V\to W$. 
We denote by: 
\begin{enumerate}[]
\item ${\cal L}(V,W)$ the space of all linear operators from $V$ to $W$;  
\item ${\cal L}_r(V,W)$ the space of all regular operators from $V$ to $W$;  
\item ${\cal L}_{ob}(V,W)$ the space of all order bounded operators from $V$ to $W$.  
\end{enumerate}
\noindent
Clearly, ${\cal L}_r(V,W)\subseteq{\cal L}_{ob}(V,W)$.
The order dual $E^\sim$ of $E$ coincides with the norm dual $E'$.
For any $S\in{\cal L}_{ob}(E,F)$, the adjoint $S':F'\to E'$ is order bounded and o-continuous     
(cf. \cite[Thm.1.73]{AlBu}). Furthermore, let $T\in{\cal L}_{ob}(E,F)$, 
then $T''\in{\cal L}_{ob}(E'',F'')$. 
Since $F''$ is Dedekind complete, 
$T''\in{\cal L}_r(E'',F'')\subseteq\text{\rm L}(E'',F'')$. 
Thus $T\in\text{\rm L}(E,F)$, and hence 
${\cal L}_{ob}(E,F)\subseteq\text{\rm L}(E,F)$.
We shall also use   
\begin{equation}\label{nf}
   \|f\|=\sup\{|f(x)|: x\in B_X\}=\sup\{|y(f)|: y\in B_{X''}\} \ \ \ \ (\forall f\in X')
\end{equation}
and, as a consequence of (\ref{nf}),   
\begin{equation}\label{convbb}
  (f_n)\ \text{\rm in}\ X'\ \text{\rm converges uniformly on}\ B_X\ \text{\rm iff it converges uniformly on}\ B_{X''} 
\end{equation}
under the identification of $f\in X'$ with $\hat{f}\in(X')''$.
By \cite[Def.1.1]{AAT}: a subset $A\subseteq E$ is called b-{\em order bounded} 
whenever $\hat{A}$ is order bounded in $E''$, where $E\hat{\to}E''$ is the natural embedding 
of $E$ into $E''$. We call an operator $T:E\to F$ b-{\em order bounded}, if
$T(A)$ is order bounded in $F$ for each b-order bounded $A\subseteq E$. 
An operator $T:E\to Y$ is said to be \text{\rm b-w}{\em -compact} 
(resp. \text{\rm o-w}{\em -compact}) if 
$T$ carries \text{\rm b}-order bounded sets (resp. order interval) of $E$ onto a relatively 
\text{\rm w}-compact subset of $Y$. We denote by $\text{\rm b-W}(E,Y)$ 
(resp. by $\text{\rm o-W}(E,Y)$) the space of all \text{\rm b-w}-compact
weakly (resp. \text{\rm o-w}-compact) operators from $E$ to $Y$.  

\subsection{}
The Meyer-Nieberg theorem (cf. \cite[Thm.5.55]{AlBu}): 
\begin{enumerate}[($\ast$)]
\item
{\em if every disjoint sequence in the solid hull of $L\subseteq F$ is norm null, 
then $\text{\rm sol}(L)$ $($and hence $L$ itself$)$ is relatively w-compact in $F$} 
\end{enumerate}
produces the following proper subclass of relatively w-compact subsets.
\begin{definition}\label{LWC-subsets} 
{\em (\cite[Def.1.iii)]{Mey0})
A bounded subset $L$ of $F$ is called an \text{\rm LW}-{\em set} whenever
every disjoint sequence in $\text{\rm sol}(L)$ is norm null.}
\end{definition}
\noindent
For every \text{\rm LW}-subset $A$ of $E$, we have $A\subseteq E^a$. 
Indeed, otherwise there is $a\in A$
with $|a|\in E\setminus E^a$, and then there exists a disjoint sequence 
$(x_n)$ in $[0,|a|]\subseteq\text{\rm sol}(A)$
with $\|x_n\|\not\to 0$. 
It can be easily seen that $B_E$ is w-compact yet not an \text{\rm LW}-set
in any infinite dimensional reflexive $E$.
We need the following characterization 
of \text{\rm LW}-sets (see \cite[Prop.3.6.2]{Mey} and \cite[Lm.2.4]{BLM1}):

\begin{assertion}\label{Meyer 3.6.2}
For a nonempty bounded subset $A$ of $E$, TFAE.
\begin{enumerate}[\em (i)]
\item $A$ is an \text{\rm LW}-set.
\item For every $\varepsilon>0$, there is $u_\varepsilon\in E_+^a$ such that
$A\subseteq[-u_\varepsilon,u_\varepsilon]+\varepsilon B_E$.
\item $f_n(x_n)\to 0$ for every $(x_n)$ in $A$ and every disjoint bounded $(f_n)$ in $E'$.
\end{enumerate}
\end{assertion}
\noindent
Recall that a subset $A$ of $E$ is {\em almost order bounded} whenever, for every $\varepsilon>0$, 
there is $x\in E_+$ with $A \subseteq [-x,x]+\varepsilon B_E$. 
By Assertion \ref{Meyer 3.6.2},
\begin{equation}\label{o-bdd is Lwc}
     \text{every almost order bounded set in an \text{\rm o}-continuous $E$
     is an \text{\rm LW}-set}.  
\end{equation}
Clearly, 
\begin{equation}\label{rel comp is almost o-bdd}
    \text{every relatively compact subset of $E$ is almost order bounded}.
\end{equation}

\subsection{}
An \text{\rm LW}-{\em operator} $T:X\to F$ is defined \cite[Def.1.iii)]{Mey0}
via the condition that $T(B_X)$ is an \text{\rm LW}-set. Then 
$\text{\rm LW}(X,F)\subseteq\text{\rm W}(X,F)$, and
\begin{equation}\label{1a}
   T\in\text{\rm LW}(X,F)\Longleftrightarrow\text{\rm $T$ 
takes bounded subsets onto \text{\rm LW}-subsets.}
\end{equation}
In view of (\ref{o-bdd is Lwc}), and of (\ref{rel comp is almost o-bdd}):  
$$
   \text{if the norm on $F$ is \text{\rm o}-continuous 
   then $\text{\rm K}(X,F)\subseteq\text{\rm LW}(X,F)$}.  
$$
\noindent
The next important fact goes back to Meyer-Nieberg (see also \cite[Thm.5.63]{AlBu} and 
\cite[Prop.2.2]{BuDo} for the more general setting).
It is the key tool to various generalizations of \text{\rm LW}-operators.

\begin{assertion}\label{Burkinshaw--Dodds}
Let $A\subseteq F$ and $B\subseteq F'$ be nonempty bounded sets. TFAE.
\begin{enumerate}[\em (i)]
\item Each disjoint sequence in $\text{\rm sol}(A)$ is uniformly null on $B$.
\item Each disjoint sequence in $\text{\rm sol}(B)$  is uniformly null on $A$.
\end{enumerate}
\end{assertion}
\noindent
Assertion \ref{Burkinshaw--Dodds} yields a method of producing many clases of operators. 
For example if we consider sets of the form $A=T(C)$ with any \text{\rm w}-compact 
subset $C$ of $X$ and $B=B_{F'}$ then we obtain 
{\em almost} \text{\rm LW}- and {\em almost} \text{\rm MW}-{\em operators} \cite{BLM1}, whereas the choice of sets of the form 
$A=T[0,x]$ with any order interval $[0,x]$ in $E$ and $B=B_{F'}$ 
produces {\em order} \text{\rm LW}- and  
{\em order} \text{\rm MW}-{\em operators} \cite{BLM2}.
We include a proof of the following certainly well known fact.

\begin{lemma}\label{just lemma}
Let $L$ be a nonempty bounded subset of $F'$. TFAE.
\begin{enumerate}[\em i)]
\item $L$ is an \text{\rm LW}-subset of $F'$.
\item Each disjoint bounded sequence in $F$  is uniformly null on $L$.
\item Each disjoint bounded sequence in $F''$ is uniformly null on $L$.
\end{enumerate}
\end{lemma}

\begin{proof}
Keeping in mind (\ref{convbb}) and the fact that the unit balls in Banach lattices are solid, 
we apply Assertion \ref{Burkinshaw--Dodds} first to $A=B_F$ and $B=L$, 
and then to $A=L$ and $B=B_{F''}$ in order to obtain 
that both ii) and iii) are equivalent to the condition that 
every disjoint sequence in $\text{\rm sol}(L)$ is norm null,
which means that $L$ is an \text{\rm LW}-subset of $F'$.
\end{proof}
\noindent
The next well known assertion is a direct consequence of Lemma \ref{just lemma}.

\begin{assertion}\label{just asser}
Let $T:X\to F'$ be a bounded operator. TFAE.
\begin{enumerate}[\em (i)]
\item $T\in\text{\rm LW}(X,F')$.
\item $(T'f_n)$ is norm null for every 
disjoint bounded sequence $(f_n)$ in $F$.
\item $(T'f_n)$ is norm null for every 
disjoint bounded sequence $(f_n)$ in $F''$.
\end{enumerate}
\end{assertion}
\noindent
The condition (iii) had lead to the idea of the following definition.

\begin{definition}\label{MW operators}
{\em \cite[Def.1.iv)]{Mey0}}
{\em 
A bounded operator $S:E\to Y$ is called$:$
\begin{enumerate}[]
an \text{\rm MW}-{\em operator}, or $S\in\text{\rm MW}(E,Y)$, 
if $\|Sx_n\|\to 0$ for every disjoint bounded sequence $(x_n)$ in $E$.
\end{enumerate}}
\end{definition}
\noindent
The following duality result \cite[Satz.3]{Mey0} (cf. also \cite[Them.5.64]{AlBu}) 
has justified the success of introducing $\text{\rm LW}$- and $\text{\rm MW}$-operators.

\begin{assertion}\label{Meyer-Nieberg}
The following holds.
\begin{enumerate}[\em (i)]
\item 
  $X\stackrel{T}{\rightarrow}F$ is an \text{\rm LW}-operator $\Longleftrightarrow$ 
  $F'\stackrel{T'}{\rightarrow} X'$ is an \text{\rm MW}-operator.
\item 
  $E\stackrel{S}{\rightarrow}Y$ is an \text{\rm MW}-operator $\Longleftrightarrow$ 
$Y'\stackrel{S'}{\rightarrow}E'$ is an \text{\rm LW}-operator.
\end{enumerate}
\end{assertion}

\subsection{} 
Assertion \ref{Burkinshaw--Dodds} applied to
$A=T(C)$ with any compact $C\subseteq X$ and $B=B_{F'}$ gives:

\begin{definition}\label{c-LW operators}
{\em
\begin{enumerate}[a)]
\item 
An operator $T:X\to F$ is called a \text{\rm c-LW}-{\em ope\-rator} 
if $T$ maps compact subsets of $X$ 
into \text{\rm LW}-subsets of $F$.
\item 
A bounded operator $S:E\to Y$ is called a \text{\rm c-MW}-{\em operator} 
if $S'\in\text{\rm c-LW}(Y',E')$.
\end{enumerate}}
\end{definition}
\noindent
In view of Assertion \ref{Meyer 3.6.2}, we may replace in 
Definition \ref{c-LW operators}~a) 
compact subsets of $X$ by singletons and get:

\begin{assertion}\label{singletons}
TFAE.
\begin{enumerate}[\em (i)]
\item 
$T\in\text{\rm c-LW}(X,F)$.
\item 
$\{Tx\}$ is an $\text{\rm LW}$-subset of $F$ for each $x\in X$.
\item 
$T(X)\subseteq F^a$.
\item 
$f_n(Tx)\to 0$ for each $x\in X$ and for every disjoint bounded $(f_n)$ in $F'$.
\item 
$T'$ takes disjoint bounded sequences of $F'$ onto 
$\text{\rm w}^\ast$-null sequences of $X'$.
\end{enumerate}
\end{assertion}
\noindent
The next fact is a consequence of Lemma \ref{just lemma}.

\begin{assertion}\label{just asser c}
Let $T:X\to F'$ be a bounded operator. TFAE.
\begin{enumerate}[\em (i)]
\item $T\in\text{\rm c-LW}(X,F')$.
\item $(T'f_n)$ is $\text{\rm w}^\ast$-null in $X'$ for every 
disjoint bounded $(f_n)$ in $F$.
\item $(T'f_n)$ is $\text{\rm w}^\ast$-null in $X'$ for every 
disjoint bounded $(f_n)$ in $F''$.
\end{enumerate}
\end{assertion}
\noindent
By Assertion~\ref{singletons}, the class of \text{\rm c-LW}-operators is too 
large to be of any reasonable interest.

\bigskip
\text{\rm b-LW}-{\em operators} were introduced in \cite{BLM3}.
They correspond to taking the pair $(A,B)$ in Assertion \ref{Burkinshaw--Dodds}:
$A=T(U)$, with an arbitrary b-bounded $U\subseteq E$, and $B=B_{F'}$.
We recall the following definition.

\begin{definition}\label{bi-limited set}
{\em
A bounded subset $A$ of $X$ is called:
\begin{enumerate}[a)]
\item 
{\em limited} if each \text{\rm w}$^\ast$-null sequence in $X'$ is 
uniformly null on $A$ (cf. \cite{BD}). 
\item 
\text{\rm bi}-{\em limited} if each \text{\rm w}$^\ast$-null sequence in $X'''$ is 
uniformly null on $A$. 
\end{enumerate}}
\end{definition}
\noindent
It is well known that $B_X$ is not limited when $\dim(X)=\infty$, and 
$$
   A \ \text{\rm is compact} \ \Rightarrow A \ \text{\rm is limited} \ \Rightarrow 
   A \ \text{\rm is bi-limited}. 
$$
Thus $B_{c_0}$ is not limited in $c_0$. However $B_{c_0}$ is \text{\rm bi}-limited in $c_0$
by Phillip's lemma (cf. \cite[Thm.4.67]{AlBu}).

\bigskip
\text{\rm b-LW}-{\em operators} were introduced in \cite{BLM3}.
They correspond to taking the pair $(A,B)$ in Assertion \ref{Burkinshaw--Dodds}:
$A=T(U)$, with an arbitrary b-bounded $U\subseteq E$, and $B=B_{F'}$.

\subsection{}
We recall the following definition.

\begin{definition}\label{bi-limited set}
{\em
A bounded subset $A$ of $X$ is called:
\begin{enumerate}[a)]
\item 
{\em limited} if each \text{\rm w}$^\ast$-null sequence in $X'$ is 
uniformly null on $A$. 
\item 
\text{\rm bi}-{\em limited} if each \text{\rm w}$^\ast$-null sequence in $X'''$ is 
uniformly null on $A$. 
\end{enumerate}}
\end{definition}
\noindent
It is well known that $B_X$ is not limited when $\dim(X)=\infty$, and 
$$
   A \ \text{\rm is relatively compact} \ \Rightarrow A \ \text{\rm is limited} \ \Rightarrow 
   A \ \text{\rm is bi-limited}. 
$$
Thus $B_{c_0}$ is not limited in $c_0$. However $B_{c_0}$ is \text{\rm bi}-limited in $c_0$
by Phillip's lemma (cf. \cite[Thm.4.67]{AlBu}). Every \text{\rm bi}-limited subset of a reflexive Banach
space is relatively compact.

\subsection{Regularly ${\cal P}$-operators and the domination property.} 
The following notion was introduced in \cite[Def.2]{Emel}.
In the present paper, we develope ideas of \cite{Emel} for investigation algebraic
and lattice properties of the operators.

\begin{definition}\label{rP-operators}{\em
Let ${\cal P}$ be a set of operators (the set of ${\cal P}$-{\it operators}) 
from $V$ to $W$. An operator $T:V\to W$ is called a {\it regularly} 
${\cal P}$-{\it operator}
(shortly, an \text{\rm r}-${\cal P}$-{\it operator}), if $T$ is a difference 
of two positive ${\cal P}$-operators.}
\end{definition}
\noindent
Given a set ${\cal P}$ of operators from $V$ to $W$. We denote by: 
\begin{enumerate}[]
\item ${\cal P}(V,W)$ the set of all ${\cal P}$-operators from $V$ to $W$;  
\item ${\cal P}_+(V,W)$ the set of all positive ${\cal P}$-operators from $V$ to $W$;  
\item ${\cal P}_r(V,W)$ the set of all regular operators in ${\cal P}(V,W)$;  
\item \text{\rm r}-${\cal P}(V,W)$ the set of all 
\text{\rm r}-${\cal P}$-operators in ${\cal P}(V,W)$.
\end{enumerate}
\noindent
${\cal P}$-operators are said to satisfy the {\em domination property} if
 it follows from $0\le S\le T\in {\cal P}$
that $S$ is likewise a ${\cal P}$-operator. 
We say that an operator $T\in{\cal L}(V,W)$ is 
${\cal P}$-{\em dominated} if $\pm T\le U$ for some $U\in{\cal P}$ 
(cf. \cite[p.3022]{AW}). 

\begin{proposition}\label{prop elem}
Let ${\cal P}\subseteq{\cal L}(V,W)$, ${\cal P}\pm{\cal P}\subseteq{\cal P}\ne\emptyset$, 
and $T\in{\cal L}(V,W)$.
\begin{enumerate}[{\em i)}]
\item $T$ is an \text{\rm r}-${\cal P}$-operator 
iff $T$ is a ${\cal P}$-dominated ${\cal P}$-operator.  
\item Assume the modulus $|T|$ of $T$ exists in ${\cal L}(V,W)$ 
and suppose that ${\cal P}$-operators satisfy 
the domination property. Then $T$ is an \text{\rm r}-${\cal P}$-operator 
iff $|T|$ is a ${\cal P}$-operator.
\end{enumerate}
\end{proposition}

\begin{proof}
i) Let $T$ be an r-${\cal P}$-operator, say $T=T_1-T_2$ 
with two positive ${\cal P}$-operators $T_1,T_2$.
It follows from ${\cal P}\pm{\cal P}\subseteq{\cal P}$ that $T$ and
$U=T_1+T_2$ are both ${\cal P}$-operators. Since $\pm T\le U$ 
then $T$ is a ${\cal P}$-dominated 
${\cal P}$-operator. 

Now, let $T$ be a ${\cal P}$-dominated ${\cal P}$-operator. 
Take a ${\cal P}$-operator $U$ such that $\pm T\le U$.
Since $T=U-(U-T)$, and both $U$ and $U-T$ are 
positive ${\cal P}$-operators, $T$ is an \text{\rm r}-${\cal P}$-operator.

ii) First assume $|T|$ is a ${\cal P}$-operator. 
Since $T=T_+-T_-$, $0\le T_{\pm}\le|T|$, and 
$|T|$ is a ${\cal P}$-operator, the domination property 
implies that $T_+$ and $T_-$ are both 
positive ${\cal P}$-operators, and hence $T=T_+-T_-$ 
is an \text{\rm r}-${\cal P}$-operator. 

Now, assume $T$ is an \text{\rm r}-${\cal P}$-operator. 
Then there are two positive ${\cal P}$-operators $T_1,T_2$ 
satisfying $T=T_1-T_2$. Since $0\le T_+\le T_1$ and 
$0\le T_-\le T_2$, 
the domination property implies
that $T_+$ and $T_-$ are both ${\cal P}$-operators. 
Hence $|T|=T_++T_-$ is likewise a ${\cal P}$-operator.
\end{proof}

\begin{proposition}\label{vect lat}
Let $F$ be Dedekind complete, and ${\cal P}$ 
a subspace in $\text{\rm L}(E,F)$,
satisfying the domination property. 
Then $\text{\rm r-}{\cal P}(E,F)$
is an order ideal in the Dedekind complete vector lattice
$\text{\rm L}_r(E,F)$.
\end{proposition}

\begin{proof}
Since $F$ is Dedekind complete, $\text{\rm L}_r(E,F)$
is a Dedekind complete vector lattice.
By Proposition~\ref{prop elem}~ii), 
$T\in\text{\rm r-}{\cal P}(E,F)\Longrightarrow |T|\in\text{\rm r-}{\cal P}(E,F)$,
and hence $\text{\rm r-}{\cal P}(E,F)$ is a vector sublattice of $\text{\rm L}_r(E,F)$.
Since ${\cal P}$ satisfies the domination property, $\text{\rm r-}{\cal P}(E,F)$
is an order ideal in $\text{\rm L}_r(E,F)$.
\end{proof}

In general, quite a little is known about 
conditions on ${\cal P}$-operators under that 
every regular ${\cal P}$-operator is an \text{\rm r}-${\cal P}$-operator, even for compact operators in a Banach lattice (see \cite{CW0,Cheng} 
and references therein). By \cite[Thm.4(ii)]{AW}, 
there exists a compactly dominated compact operator 
on $L^2[0,1]\oplus c(L^2[0,1])$ without modulus. 
So, in terms of Proposition \ref{prop elem}~i), 
an \text{\rm r}-compact operator need not to have modulus. 
By \cite[Thm.2.2]{CW1},  
\text{\rm r}-compact (resp. \text{\rm r-w}-compact, 
\text{\rm r-LW}-, \text{\rm r-MW}-) operators from 
$E=L^2[0,1]$ to $F=c(L^2[0,1])$ form a proper 
subspace of the space of regular compact 
(resp. regular \text{\rm w}-compact, 
regular \text{\rm LW}-, regular \text{\rm MW}-) operators from $E$ to $F$. 
In view of Proposition~\ref{prop elem}~ii), the following fact 
is a direct consequence 
of \cite[Satz 7.1]{Kre2} and \cite[Thm.2.9]{CW1}.
\noindent
\begin{enumerate}[]
\item If $F$ is an \text{\rm AM}-space, then every compact 
(\text{\rm LW}-) operator 
from $E$ to $F$ is \text{\rm r}-compact (resp. \text{\rm r-LW}-).
\end{enumerate}
\noindent
The domination property for compact, \text{\rm w}-compact, 
\text{\rm LW}-, \text{\rm MW}-, and Dunford--Pettis
operators was well investigated in the last quarter of 20th century.
Recently, it was studied in \cite{AEG2,Emel,EG} 
for KB and Levi operators between locally solid lattices.

The following example goes back to \cite{Kre1} 
(see, also \cite[Ex.5.6]{AlBu},\cite{AW},\cite{AN}). 

\begin{example}\label{Krengel}
{\em Let $A_n$ be a $2^n\times 2^n$ matrix constructed inductively:
$$
A_1=\left[
\begin{matrix}
1 & 1 \\
1 & -1
\end{matrix}
\right] \quad \text{\rm and} \quad A_{n+1}=\left[
\begin{matrix}
A_n & A_n \\
A_n & -A_n
\end{matrix}
\right].
$$
For each $n$,  let $T_n:\ell^2_{2^n}\to\ell^2_{2^n}$ 
be the operator,  whose matrix (with respect to
the standard unit vectors of 
$\ell^2_{2^n}= {\mathbb R}^{2^n}$) is $2^{-\frac{n}{2}}A_n$.
Since the matrix $2^{-\frac{n}{2}}A_n$ is orthogonal, each $T_n$ is an isometry. 
On the other hand, $|T_n|$ is the operator, whose matrix is a $2^n\times 2^n$ 
matrix
with all of its entries equal to 1, and hence $\||T_n|\|=2^{\frac{n}{2}}$ 
for all $n$.
Consider the $c_0$-direct sum  
$E:=(\oplus_{n=1}^{\infty}\ell^2_{2^n})_0$. 
Note that $E'=(\oplus_{n=1}^{\infty}\ell^2_{2^n})_1$, 
in particular, $E$ and $E'$ both have $\text{\rm o}$-continuous norms. 
By \cite[Thm.4.5]{DF}, 
compact operators on $E$ satisfy the domination property.
For any $\alpha=(\alpha_n)\in\ell^\infty$, define $T:E\to E$, by
\begin{equation}\label{1kr}
   Tx=T(x_1,x_2,\dots):=(\alpha_1T_1x_1,\alpha_2T_2x_2,\dots).
\end{equation}
Then $\|T\|\le\|\alpha\|_\infty$, and if
$\alpha_n\in c_0$, then $T\in\text{\rm K}(E)$. If $\alpha_n=2^{-\frac{n}{2}}$, 
the modulus $|T|$ exists and is given by $|T|x=(\alpha_1|T_1|x_1,\alpha_2|T_2|x_2,\dots)$.
As $|T|$ is not compact, $T\in\text{\rm K}_r(E)\setminus\text{\rm r-K}(E)$
by Proposition \ref{prop elem}~ii).
For $\alpha_n=2^{-\frac{n}{3}}$, $T$ is compact but not order bounded;
so, $|T|$ does not exist.}
\end{example}

\subsection{Disjoint sequence properties.}
In this subsection we discuss some facts and results 
on Banach lattices used in the paper. The spaces of operators 
that we study here as well as the spaces serving 
as domains and/or ranges to the operators have 
many interesting disjointness properties.
We begin with the next well known fact 
(cf. \cite[Thm.4.59]{AlBu}; \cite[2.4.14]{Mey}) 
is used freely in the present paper. 

\begin{assertion}\label{E' is o-cont} 
TFAE.
\begin{enumerate}[\em (i)] 
\item The norm on $E'$ is \text{\rm o}-continuous.
\item $E'$ is a KB space.
\item Every disjoint bounded sequence in $E$ is \text{\rm w}-null.
\end{enumerate}
\end{assertion} 

Recall that $E$ has {\em sequentially} w-{\em continuous} 
(w$^\ast$-{\em continuous}) {\em lattice operations} if, for every 
w-null $(x_n)$ in $E$ (w$^\ast$-null $(x_n)$ in $E'$), 
the sequence $(|x_n|)$ is also w-null 
(w$^\ast$-null). It follows from \cite[Thm.4.34]{AlBu}) that,
for any disjoint $\text{\rm w}$-null $(x_n)$ in $E$, 
the sequence $(|x_n|)$ is \text{\rm w}-null.

\begin{definition}\label{d-property}
{\em (\cite[Def.1]{El}). $E$ is said to have the {\em property} (d) 
(shortly, $E\in\text{\rm (d)})$ if $(|f_n|)$ is also 
$\text{\rm w}^\ast$-null for each disjoint $\text{\rm w}^\ast$-null $(f_n)$ in $E'$.}
\end{definition}
\noindent
The property (d) is weaker than the sequential $\text{w}^\ast$-continuity of the lattice operations, 
as $E=\ell^{\infty}$ shows. It was already studied in \cite{Wnuk3}, 
where no name was assigned to this property. In particular, it was proved in \cite[Prop.1.4]{Wnuk3} 
that every $\sigma$-Dedekind complete Banach lattice $E$ has the property (d), and it was  
observed in \cite[Rem.1.5]{Wnuk3} that $\ell^{\infty}/c_0\in\text{\rm (d)}$, 
but $\ell^{\infty}/c_0$ is not $\sigma$-Dedekind complete. 
Example~\ref{DPSP but not DSP} below shows that $C[0,1]\not\in\text{\rm (d)}$.

\begin{definition}\label{Schur property}
{\em A Banach lattice $E$ is said to have:
\begin{enumerate}[a)]
\item 
the {\em Schur property} (shortly, $E\in(\text{\rm SP})$)
if \text{\rm w}-null sequences in $E$ are norm null (cf. \cite[p.207]{AlBu});
\item 
the {\em positive Schur property} (shortly, $E\in(\text{\rm PSP})$)
if positive \text{\rm w}-null sequences in $E$ are norm null (cf. \cite{Wnuk3});
\item 
the {\em dual positive Schur property} (shortly, $E\in(\text{\rm DPSP})$)
if positive w$^\ast$-null sequences in $E'$ are norm null \cite[Def.3.3]{AEW};
\item 
the {\em dual disjoint Schur property} (shortly, $E\in(\text{\rm DDSP})$)
if disjoint w$^\ast$-null sequences in $E'$ are norm null;
\item 
the {\em Grothendieck property} (shortly, $E\in(\text{\rm GP})$) 
if w$^\ast$-null sequences in $E'$ are w-null (cf. \cite[p.760]{Wnuk3});
\item 
the {\em positive Grothendieck property} (shortly, $E\in(\text{\rm PGP})$) 
if positive w$^\ast$-null sequences in $E$ are w-null (cf. \cite[p.760]{Wnuk3}); 
\item 
the {\em disjoint Grothendieck property} (shortly, $E\in(\text{\rm DGP})$) 
if disjoint w$^\ast$-null sequences in $E'$ are w-null;
\item 
the {\em dual disjoint \text{\rm w}$^\ast$-property} (shortly, $E\in(\text{\rm DDw$^\ast$P})$) 
if each disjoint bounded sequence in $E'$ is w$^\ast$-null. 
\end{enumerate}}
\end{definition}
\noindent
The disjoint Grothendieck property was introduced under the name the {\em weak Grothendieck
property} in \cite[Def.4.8)]{MFMA}. 

\begin{proposition}\label{DDw*P}
For every Banach lattice $E$ the following holds.
$$
   E\in(\text{\rm DGP})\cap(\text{\rm DDw}^\ast\text{\rm P})\Longrightarrow E'' 
   \text{\it \ has \text{\rm o}-continuous norm}\Longrightarrow 
   E\in(\text{\rm DDw}^\ast\text{\rm P}).
$$
\end{proposition}
\begin{proof}
Let $E\in(\text{\rm DGP})\cap(\text{\rm DDw$^\ast$P})$ but the norm in $E''$ 
is not \text{\rm o}-continuous.
Then $(f_n)$ is not \text{\rm w}-null for some disjoint bounded $(f_n)$ in $E'$
by Assertion \ref{E' is o-cont}. However $f_n\stackrel{\text{\rm w}^\ast}{\to}0$
because $E\in(\text{\rm DDw$^\ast$P})$, and hence 
$f_n\stackrel{\text{\rm w}}{\to}0$ due to $E\in(\text{\rm DGP})$.
The obtained contradiction proves the first implication.

For the second implication, let the norm in $E''$ be \text{\rm o}-continuous. 
Then, by Assertion \ref{E' is o-cont}, every disjoint bounded $(f_n)$ in $E'$ 
is \text{\rm w}-null and hence $\text{\rm w}^\ast$-null.  
It follows $E\in(\text{\rm DDw$^\ast$P})$.
\end{proof}

The next important fact was proved in \cite[Prop.2.3]{Wnuk3}. 
We do not know whether or not 
in this assertion (DPSP) can be replaced by (PGP) and the norm null condition by 
the w-null condition. 

\begin{assertion}\label{Wnuk2013}
$E\in \text{\rm (DPSP)}$ iff every disjoint \text{\rm w}$^\ast$-null sequence in $E'_+$ is norm null.
\end{assertion}
\noindent
Clearly, (DPSP)$\Longrightarrow$(PGP); and (DDSP)$\Longrightarrow$(DGP). In AM-spaces, (PGP) agrees with (DPSP),
by \cite[Prop.4.1]{Wnuk3}. (PSP) and (SP) coincide in the class of discrete Banach lattices \cite[p.19]{Wnuk1}.
By \cite[Thm.5.3.13]{Mey}, (PGP) and (GP) agree in the class of Banach lattices with the interpolation property. 

The following assertion \cite[Thm.3.3]{MEM} describes 
the impact of the property~(d) on dual Schur properties.

\begin{assertion}\label{DSP and (d)}
For every Banach lattice $E$ the following holds.
\begin{enumerate}[{\em i)}]
\item
$E\in(\text{\rm DDSP}) \ \Longrightarrow \ E\in(\text{\rm DPSP})$$;$
\item 
if $E$ has the property~{\em (d)} then\\ 
$E\in(\text{\rm DDSP}) \ \Longleftrightarrow \ E\in(\text{\rm DPSP})$.
\end{enumerate}
\end{assertion}

%\begin{proof}
%i) Let $E\in(\text{\rm DDSP})$ and let $(f_n)$ be a disjoint 
%positive w$^\ast$-null sequence in $E'$.
%By Definition~\ref{Schur property}\,d), $(f_n)$ is norm null. 
%Since the disjoint positive w$^\ast$-null sequence
%$(f_n)$ in $E'$ was taken arbitrary, Assertion~\ref{Wnuk2013} 
%implies that $E\in(\text{\rm DPSP})$.
%
%ii) Assume now $E\in\text{\rm (d)}$.  It suffices to prove $(\Longleftarrow)$.
%Let $E\in(\text{\rm DPSP})$ and let $(f_n)$ be 
%a disjoint w$^\ast$-null in $E'$. 
%Then $(|f_n|)$ is a disjoint positive w$^\ast$-null sequence in $E'$ 
%by the property~(d) of $E$. 
%Since $E\in\text{\rm (DPSP)}$, $(|f_n|)$ is norm null. 
%Then $(f_n)$ is also norm null, and hence $E\in(\text{\rm DDSP})$, as required. 
%\end{proof}
\noindent
Clearly $(\text{PGP})\Longrightarrow (\text{DGP})$ for Banach lattices with the property~(d).
For this implication, the property~(d) is essential, as $C[0,1]\in\text{\rm (PGP)}\setminus\text{\rm (DGP)}$
by Example~\ref{DPSP but not DSP} below.
We have no example of $E\in(\text{DGP})\setminus(\text{PGP})$.
\noindent
For the next assertion, we refer to \cite[Cor.3.6.8]{Mey}, \cite[Thm.7]{Wnuk1}.

\begin{assertion}\label{Schur}
For a Banach lattice $E$, TFAE.
\begin{enumerate}[\em (i)]
\item $E\in(\text{\rm PSP})$.
\item Each disjoint $\text{\rm w}$-null sequence in $E$ is norm null.
\item Each disjoint $\text{\rm w}$-null sequence in $E_+$ is norm null.
\item Every disjoint sequence in the solid hull of each relatively $\text{\rm w}$-compact subset of $E$ is norm-null.
\item Almost order bounded subsets of $E$ coincide with relatively $\text{\rm w}$-compact subsets of $E$.
\end{enumerate}
\end{assertion}
\noindent
By Assertion~\ref{Schur}, 
$
\text{\rm if} \   F\in\text{\rm (PSP)}\ \text{\rm then} \  
\text{\rm W}(X,F)\subseteq\text{\rm LW}(X,F).
$
\noindent
The next example shows that the property~(d) is essential 
in Assertion~\ref{DSP and (d)}(ii). 

\begin{example}\label{DPSP but not DSP}
{\em Consider the Banach lattice $E=C[0,1]$. Let $(f_n)$ be a positive w$^\ast$-null sequence in $E'$.
It follows from $\|f_n\|=f_n({\mathbb 1}_{[0,1]})\to 0$ that $C[0,1]\in(\text{\rm DPSP})$, 
and hence $C[0,1]\in(\text{\rm PGP})$.

Take the sequence $(r_n)$ of the Rademacher functions on $[0,1]$ and define a   
disjoint w$^\ast$-null sequence in $E'$ as follows:
$g_n=2^{n+1}\cdot r_{2n}\cdot {\mathbb 1}_{[\frac{1}{2^{n+1}},\frac{1}{2^n}]}$.
Then $(g_n)$ is disjoint w$^\ast$-null sequence in $E'$, 
yet $\|g_n\|=1$ for all $n\in {\mathbb N}$.
Therefore $C[0,1]\not\in(\text{DDSP})$, and hence 
by Assertion~\ref{DSP and (d)}, 
$C[0,1]\not\in(\text{d})$.
Note that $(g_n)$ is even not w-null. Indeed, for 
$
  y=\sum_{n=1}^\infty r_{2n}\cdot 
  {\mathbb 1}_{[\frac{1}{2^{n+1}}, \frac{1}{2^n}]} \in
  L^\infty[0,1] \subseteq E',
$
$y(g_n)\equiv 1$ for all $n\in {\mathbb N}$.
It also shows that $C[0,1]\not\in(\text{DGP})$.}
\end{example}
\noindent
For further unexplained terminology and notations, we refer to \cite{AlBu,AAT,AEG1,AEG2,BHM,BLM2,BLM1,EAS,Mey}.

%%%%%%%%%%%%%%%%%%%%
\section{\text{\rm a-LW}- and \text{\rm l-LW}-operators}
%%%%%%%%%%%%%%%%%%%%

Almost \text{\rm LW}- and almost \text{\rm MW}-operators were introduced in \cite{BLM1}. 
Here, apart from investigating them, we introduce and study limitedly \text{\rm LW}- 
and \text{\rm MW}-operators. They both respectively generalize \text{\rm LW}- and 
\text{\rm MW}-operators. We also show below in Proposition \ref{l-LW vs DDw*P}
that the property 
$\text{\rm DDw$^\ast$P}$ puts an end to further
extensions as when $F\in(\text{\rm DDw$^\ast$P})$
then $\text{\rm l-LW}(X,F)=\text{\rm L}(X,F)$ for any Banach space $X$.

\subsection{} 
We begin with the following definition.

\begin{definition}\label{Main LWC operators}
{\em An operator $T:X\to F$ is called:
\begin{enumerate}[i)]
\item 
{\em almost $\text{\rm LW}$} ($T$ is \text{\rm a-LW}), 
if $T$ carries relatively w-compact subsets of $X$ onto \text{\rm LW}-subsets of $F$ \cite[Def.2.1]{BLM1};
\item 
{\em limitedly $\text{\rm LW}$} ($T$ is \text{\rm l-LW}), 
if $T$ carries limited subsets of $X$ onto \text{\rm LW}-subsets of $F$.
\item 
\text{\rm bi}-{\em limitedly $\text{\rm LW}$} ($T$ is \text{\rm bi-l-LW}), 
if $T$ carries \text{\rm bi}-limited subsets of $X$ onto \text{\rm LW}-subsets of $F$.
\end{enumerate}}
\end{definition}
\noindent
Clearly, 
$\text{\rm LW}(X,F)\subseteq\text{\rm bi-l-LW}(X,F)\subseteq
\text{\rm l-LW}(X,F)\subseteq\text{\rm c-LW}(X,F)$.
Furthermore, the above inclusion turn to identities whenever $X$ is reflexive.

\begin{example}\label{l-LW operator that is not bi-l-LW}
{\em 
It was pointed out after Definition \ref{bi-limited set} that $B_{c_0}$ is not limited yet
is \text{\rm bi}-limited in $c_0$. Since $B_{c_0}$ is not an \text{\rm LW}-set in $c_0$,
$I_{c_0}\notin\text{\rm bi-l-LW}(c_0)$. However, $I_{c_0}\in\text{\rm l-LW}(c_0)$ since
each limited subset of $c_0$ is relatively compact and relatively compact subsets of $c_0$
are \text{\rm LW}-sets.}
\end{example}

\begin{definition}\label{bi-limited set}
{\em
A bounded subset $A$ of $X$ is called:
\begin{enumerate}[a)]
\item 
{\em limited} if each \text{\rm w}$^\ast$-null sequence in $X'$ is 
uniformly null on $A$. 
\item 
\text{\rm bi}-{\em limited} if each \text{\rm w}$^\ast$-null sequence in $X'''$ is 
uniformly null on $A$. 
\end{enumerate}}
\end{definition}
\noindent
It is well known that $B_X$ is not limited when $\dim(X)=\infty$, and 
$$
   A \ \text{\rm is relatively compact} \ \Rightarrow A \ \text{\rm is limited} \ \Rightarrow 
   A \ \text{\rm is bi-limited}. 
$$
Thus $B_{c_0}$ is not limited in $c_0$. However $B_{c_0}$ is \text{\rm bi}-limited in $c_0$
by Phillip's lemma (cf. \cite[Thm.4.67]{AlBu}). Every \text{\rm bi}-limited subset of a reflexive Banach
space is relatively compact.

\subsection{} 
We remind several facts on \text{\rm a-LW}-operators.
As w-compact sets are bounded, 
$\text{\it each \text{\rm LW}-operator is an \text{\rm a-LW}-operator}$.
\noindent
By \eqref{1a}, if $S\in\text{\rm W}(Y,X)$ and $T\in\text{\rm a-LW}(X,F)$ then
$T  S\in\text{\rm LW}(Y,F)$.
Furthermore, each \text{\rm a-LW}-operator $T:X\to F$ is automatically bounded. Indeed, otherwise
there exists a norm null sequence $(x_n)$ in $B_X$ satisfying $\|Tx_n\|\ge n$
for all $n\in\mathbb{N}$, violating the fact that the \text{\rm LW}-set 
$\{Tx_n: n\in\mathbb{N}\}$ is bounded. Since $\text{\rm LW}(X,F)\subseteq\text{\rm W}(X,F)$,
the identity operator $I_{\ell^1}:\ell^1\to\ell^1$ is not \text{\rm LW}. 
It was shown in \cite[Prop.1]{EAS} that
$T\in\text{\rm a-LW}(X,F)$ iff $T(X)\subseteq F^a$ and $f_n(Tx_n)\to 0$ 
for each disjoint bounded $(f_n)$ in $F'$ and each w-null $(x_n)$ in $X$.

\bigskip 
\noindent
Since ${\ell^1}$ has the Schur property, 
$I_{\ell^1}\in\text{\rm a-LW}(\ell^1)\setminus\text{\rm W}(\ell^1)$. 
It was proved in \cite[Thm.2.4]{BLM1}, that 
\begin{equation}\label{3a}
   T\in\text{\rm a-LW}(X,F) \ \text{\rm iff} \ T  S\in\text{\rm LW}(Y,F) \ 
   \text{\rm for every} \ Y \ \text{\rm and all} \ S\in\text{\rm W}(Y,X).
\end{equation}
\begin{lemma}\label{Composition of LW}
For a pair of operators $T$ and $S$$:$ 
\begin{enumerate}[\em (i)]
\item if $S\in\text{\rm LW}(X,E)$ and $T\in\text{\rm a-LW}(E,F)$ then $T  S\in\text{\rm LW}(X,F)$$;$
\item if $S\in\text{\rm L}(X,Y)$ and $T\in\text{\rm a-LW}(Y,F)$ then $T  S\in\text{\rm a-LW}(Y,F)$.
\end{enumerate}
\end{lemma}
\begin{proof}
(i) follows from \eqref{3a} because \text{\rm LW}-operators are \text{\rm w}-compact. 

(ii) Because every bounded operator is \text{\rm w}-continuous.
\end{proof}

\begin{proposition}\label{LW-algebra}
For any Banach lattice $E$ the following hold.
\begin{enumerate}[{\rm i)}]
\item $\text{\rm LW}(E)$ is a closed subalgebra of $\text{\rm W}(E)$ and is unital 
iff the identity operator $I_E$ is \text{\rm LW}.
\item $\text{\rm a-LW}(E)$ is a closed right ideal in $\text{\rm L}(E)$ 
$($and hence a subalgebra of $\text{\rm L}(E)$$)$, and is unital iff $I_E$ is {\em a-LW}.
\end{enumerate}
\end{proposition}
\begin{proof}
It is easy to see that both $\text{\rm LW}(E)$ and $\text{\rm a-LW}(E)$ are 
closed subspaces of $\text{\rm L}(E)$ 
(see, for example, \cite[Prop.2.1]{BLM1}), and $\text{\rm LW}(E)\subseteq\text{\rm W}(E)$. 
The closeness of $\text{\rm LW}(E)$ under composition follows from Lemma \ref{Composition of LW}(i)(ii).
The fact that $\text{\rm a-LW}(E)$ is a right ideal in $\text{\rm L}(E)$ follows from Lemma \ref{Composition of LW}(ii).
The conditions on $I_E$, those make the algebras $\text{\rm LW}(E)$ and $\text{\rm a-LW}(E)$ unital,
are simply verified.
\end{proof}
\noindent
As $\text{\rm a-LW}(E,F)$ satisfies the domination property by \cite[Thm.1]{AkGo},
the next lemma follows from Proposition~\ref{prop elem}~ii).

\begin{lemma}\label{prop elem aLW}
Let an operator $T:E\to F$ possess the modulus. TFAE.
\begin{enumerate}[{\em i)}]
\item $T$ is a regularly $\text{\rm a-LW}$-operator.
\item $|T|$ is an $\text{\rm a-LW}$-operator.
\end{enumerate}
\end{lemma}
\noindent 
We also need the following fact \cite[Lm.1]{Emel}.

\begin{assertion}\label{P-norm}
Let ${\rm P}$ be a closed subspace of $\text{\rm L}(E)$ in operator norm.
Then the formula 
$\|T\|_\text{\rm r-P}=\inf\{\|S\|:\pm T\le S\in{\rm P}\}$
defines a norm on the space $\text{\rm r-P}(E)$. Furthermore,
$\|T\|_{\text{\rm r-P}}\ge\|T\|_r\ge\|T\|$ for all $T\in\text{\rm r-P}(E)$, and  
$(\text{\rm r-P}(E),\|\cdot\|_{\text{\rm r-P}})$ \ is a Banach space.
\end{assertion}

\begin{theorem}\label{a-LW-Riesz-algebra}
For a Banach lattice $E$ the following hold. 
\begin{enumerate}[{\em i)}]
\item 
The space $\text{\rm r-a-LW}(E)$ of regularly 
$\text{\rm a-LW}$-operators on $E$ is a subalgebra 
of $\text{\rm L}_r(E)$. Moreover, 
$
   \text{\rm r-a-LW}(E)=\text{\rm L}_r(E) \ \Longleftrightarrow I_E\in\text{\rm a-LW}(E).
$ 
\item 
If, additionally $E$ is Dedekind complete then 
$(\text{\rm r-a-LW}(E), \ \|\cdot\|_{\text{\rm r-a-LW}})$
is a closed Riesz subalgebra and an order ideal of the Banach lattice algebra 
$(\text{\rm L}_r(E), \ \|\cdot\|_r)$.
\end{enumerate}
\end{theorem}

\begin{proof}
i) Since regular operators on $E$ are bounded, it follows from Lemma~\ref{Composition of LW}(ii)
that $\text{\rm r-a-LW}(E)$ is a right algebra ideal and hence 
is a subalgebra of $\text{\rm L}_r(E)$. 
The condition on $I_E$ for 
$\text{\rm r-a-LW}(E)=\text{\rm L}_r(E)$ is obvious.

ii) It follows from Lemma~\ref{prop elem aLW} 
that the subalgebra $\text{\rm r-a-LW}(E)$ 
is closed under lattice operations in $\text{\rm L}_r(E)$, and hence 
is a Riesz subalgebra of $\text{\rm L}_r(E)$. Proposition \ref{vect lat} implies 
that $\text{\rm r-a-LW}(E)$ is an order ideal of $\text{\rm L}_r(E)$.

Let $S,T\in\text{\rm r-a-LW}(E)$ satisfy $|S|\le|T|$.
In order to show $\|S\|_{\text{\rm r-a-LW}(E)}\le\|T\|_{\text{\rm r-a-LW}(E)}$,
observe that $|S|,|T|\in\text{\rm a-LW}(E)$ by Lemma \ref{prop elem aLW}. Then 
$$
   \|S\|_{\text{\rm r-a-LW}(E)}=\|~|S|~\|\le\|~|T|~\|=\|T\|_{\text{\rm r-a-LW}(E)}.
$$
In particular $\|\cdot\|_{\text{\rm r-a-LW}(E)}$ coincide on $\text{\rm r-a-LW}(E)$ 
with the regular norm $\|\cdot\|_r$, 
since $\|T\|_r=\|~|T|~\|$ for every $T\in\text{\rm L}_r(E)$ \cite[p.255]{AlBu}.

By Assertion \ref{P-norm}, it follows from \cite[Prop.2.1]{BLM1}
that $(\text{\rm r-a-LW}(E), \ \|\cdot\|_{\text{\rm r-a-LW}})$ is a Banach space.
It remains to show that the norm $\|\cdot\|_{\text{\rm r-a-LW}}$
is submultiplicative. Let $S,T\in{\text{\rm r-a-LW}}(E)$.
Since $\pm S\le|S|\in\text{\rm a-LW}(E)$ and $\pm T\le|T|\in\text{\rm a-LW}(E)$ then
$\pm ST\le|ST|\le|S|\cdot|T|$, and hence 
$$
   \|S  T\|_{\text{\rm r-a-LW}(E)}\le\|~|S||T|~\|\le
   \|~|S|~\|\cdot\|~|T|~\|=
\|S\|_{\text{\rm r-a-LW}(E)}\cdot\|T\|_{\text{\rm r-a-LW}(E)},
$$
as desired.
\end{proof}

\subsection{l-LW-operators.} 
Here we include basic properties of \text{\rm l-LW}-operators.

\begin{lemma}\label{l-LW-operators}
TFAE.
\begin{enumerate}[\em (i)]
\item $T\in\text{\rm l-LW}(X,F)$.
\item $T'f_n\stackrel{\text{\rm w}^\ast}{\to}0$ in $X'$ for each disjoint bounded $(f_n)$ in $F'$.
\end{enumerate}
\end{lemma}
\begin{proof}
(i)$\Longrightarrow$(ii): Let $T\in\text{\rm l-LW}(X,F)$, $(f_n)$ be disjoint bounded in $F'$,
and $x\in X$. As $\{Tx\}$ is an $\text{\rm LW}$-subset of $F$ then
$T'f_n(x)=f_n(Tx)\to 0$ by Assertion \ref{Burkinshaw--Dodds}. Since $x\in X$
is arbitrary, $(T'f_n)$ is $\text{\rm w}^\ast$-null.

(ii)$\Longrightarrow$(i): Let $T'f_n\stackrel{\text{\rm w}^\ast}{\to}0$ 
for each disjoint bounded $(f_n)$ in $F'$. Suppose
$T\not\in\text{\rm l-LW}(X,F)$. So, $T(L)$ is not an $\text{\rm LW}$-subset of $F$
for some nonempty limited subset $L$ of $X$. Then there exists a disjoint sequence
$(g_n)$ in $B_{F'}$, and it is not uniformly null on $T(L)$ by Assertion \ref{Burkinshaw--Dodds}.
Therefore, $(T'g_n)$ is not uniformly null on $L$ violating that
$T'g_n\stackrel{\text{\rm w}^\ast}{\to}0$ and that $L$ is limited subset of $X$.
The obtained contradiction proves the implication.
\end{proof}

\begin{proposition}\label{l-LW-domination}
Let $0\le S\le T\in\text{\rm l-LW}(E,F)$ then $S\in\text{\rm l-LW}(E,F)$.
\end{proposition}
\begin{proof}
Let $(f_n)$ be disjoint bounded in $F'$. Then $(|f_n|)$ is also disjoint bounded,
and hence $T'|f_n|\stackrel{\text{\rm w}^\ast}{\to}0$ by Lemma \ref{l-LW-operators}.
It follows from
$$
   |S'f_n(x)|\le S'|f_n|(|x|)\le T'|f_n|(|x|)\to 0  \ \ \ \ (\forall x\in E)
$$
that $S'f_n\stackrel{\text{\rm w}^\ast}{\to}0$. Using Lemma \ref{l-LW-operators} again, we conclude $S\in\text{\rm l-LW}(E,F)$.
\end{proof}

\begin{lemma}\label{l-LW-closed}
Let $\text{\rm l-LW}(X,F)\ni T_n\stackrel{\|\cdot\|}{\to} T$. Then $T\in\text{\rm l-LW}(X,F)$.
\end{lemma}
\begin{proof}
Let $(f_n)$ be disjoint bounded in $F'$, and $x\in X$. By Lemma \ref{l-LW-operators},
we need to show $T'f_n(x)\to 0$. Let $\varepsilon>0$. Pick any $k\in\mathbb{N}$
with $\|T-T_k\|\le\varepsilon$. Since $T_k\in\text{\rm l-LW}(X,F)$  
then $|T_k'f_n(x)|\le\varepsilon$ for $n\ge n_0$. As $\varepsilon>0$ is arbitrary,
it follows from
$$
   |T'f_n(x)|\le|T'f_n(x)-T'_kf_n(x)|+|T_k'f_n(x)|\le 
$$
$$
   \|T'-T'_k\|\|f_n\|\|x\|+|T_k'f_n(x)|\le(\|f_n\|\|x\|+1)\varepsilon 
   \ \ \ \ (\forall n\ge n_0)
$$
that $T'f_n(x)\to 0$, as desired.
\end{proof}

\begin{proposition}\label{l-LW vs DDw*P}
{\em TFAE.
\begin{enumerate}[(i)]
\item $F\in(\text{\rm DDw$^\ast$P})$.
\item $I_F\in\text{\rm l-LW}(F)$.
\item Each limited subset of $F$ is an $\text{\rm LW}$-set.
\item $\text{\rm l-LW}(F)=\text{\rm L}(F)$.
\item $\text{\rm l-LW}(X,F)=\text{\rm L}(X,F)$ for each Banach space $X$.
\end{enumerate}}
\end{proposition}
\begin{proof}
$\text{\rm (i)}\Longleftrightarrow\text{\rm (ii)}$: 
It follows from Lemma \ref{l-LW-operators}.

The implications $\text{\rm (v)}\Longrightarrow\text{\rm (iv)}\Longrightarrow
\text{\rm (ii)}\Longleftrightarrow\text{\rm (iii)}$ are trivial.

$\text{\rm (iii)}\Longrightarrow\text{\rm (v)}$: Let $T\in\text{\rm L}(X,F)$ and let
$L$ be limited subset of $X$. Then $T(L)$ is a limited subset of $F$, and hence 
$T(L)$ is an $\text{\rm LW}$-subset of $F$. Thus, $T\in\text{\rm l-LW}(X,F)$,
as desired.
\end{proof}

\begin{proposition}\label{l-LW-algebra}
$\text{\rm l-LW}(E)$ is a closed right ideal in $\text{\rm L}(E)$ 
$($and hence a subalgebra of $\text{\rm L}(E)$$)$, and it is unital iff 
$I_E$ is \text{\rm l-LW}.
\end{proposition}
\begin{proof}
$\text{\rm l-LW}(E)$ is a closed subspace of $\text{\rm L}(E)$ by 
Lemma \ref{l-LW-closed}. 
Since each bounded operator maps limited sets onto limited sets, $\text{\rm l-LW}(E)$ 
is a right ideal in $\text{\rm L}(E)$.
The condition on $I_E$ making algebra $\text{\rm l-LW}(E)$ unital is trivial.
\end{proof}
\noindent
In view of Proposition \ref{l-LW-domination},
the next lemma follows from Proposition~\ref{prop elem}~ii).

\begin{lemma}\label{prop elem l-LW}
Let an operator $T:E\to F$ possess the modulus. TFAE.
\begin{enumerate}[{\em i)}]
\item $T$ is a regularly $\text{\rm l-LW}$-operator.
\item $|T|$ is an $\text{\rm l-LW}$-operator.
\end{enumerate}
\end{lemma}

\begin{theorem}
The following statements hold. 
\begin{enumerate}[{\em i)}]
\item 
$\text{\rm r-l-LW}(E)$ is a subalgebra of $\text{\rm L}_r(E)$. Moreover, 
$$
   \text{\rm r-l-LW}(E)=\text{\rm L}_r(E) \ \Longleftrightarrow I_E\in\text{\rm l-LW}(E).
$$ 
\item 
If $E$ is Dedekind complete then 
$(\text{\rm r-l-LW}(E), \ \|\cdot\|_{\text{\rm r-l-LW}})$
is a closed Riesz subalgebra and an order ideal of $(\text{\rm L}_r(E), \ \|\cdot\|_r)$.
\end{enumerate}
\end{theorem}

\begin{proof}
i) 
It follows from Proposition \ref{l-LW-algebra}
that $\text{\rm r-l-LW}(E)$ is a right ideal and hence is 
a subalgebra of $\text{\rm L}_r(E)$. 
The condition on $I_E$ under that $\text{\rm r-l-LW}(E)=\text{\rm L}_r(E)$ is trivial.

ii) 
It follows from Lemma \ref{prop elem l-LW} that $\text{\rm r-l-LW}(E)$ is 
a Riesz subalgebra of $\text{\rm L}_r(E)$.
The rest of the proof is similar to the proof of Theorem \ref{a-LW-Riesz-algebra}.
\end{proof}

\subsection{Notions of \text{\rm a-MW}- and \text{\rm l-MW}-operators}

\begin{definition}\label{Main MWC operators}
{\em An operator $S:E\to Y$ is called:
\begin{enumerate}[a)]
\item 
an {\em almost $\text{\rm MW}$}-operator ($S$ is \text{\rm a-MW}), 
if $f_n(Sx_n)\to 0$ for every disjoint bounded $(x_n)$ in $E$ 
and every \text{\rm w}-convergent $(f_n)$ in $Y'$ \cite[Def.2.2]{BLM1};
\item 
a {\em limitedly $\text{\rm MW}$}-operator ($S$ is \text{\rm l-MW}), 
if $Sx_n\stackrel{\text{\rm w}}{\to}0$ for every disjoint bounded $(x_n)$ in $E$.
\end{enumerate}}
\end{definition}
\begin{remark}\label{reflexive not aMW}
{\em It is easily seen that the identity operator $I_E$
\begin{enumerate}[i)]
\item 
in any infinite dimensional reflexive Banach lattice $E$ is not \text{\rm a-MW};
\item 
is an \text{\rm l-MW}-operator iff $E'$ is a KB-space.
\end{enumerate}}
\end{remark}

\subsection{}
We remind several facts on \text{\rm a-MW}-operators. The next important
semi-duality result was proved in \cite[Thm.2.5]{BLM1}.

\begin{assertion}\label{Bouras--Lhaimer--Moussa}
The following holds.
\begin{enumerate}[\em (i)]
\item 
  $T\in\text{\rm a-MW}(E,Y)\Longleftrightarrow T'\in\text{\rm a-LW}(Y',E')$. 
\item 
  $S'\in\text{\rm a-MW}(F',X')\Longrightarrow S\in\text{\rm a-LW}(X,F)$.
\end{enumerate}
\end{assertion}
\noindent
In general, $S\in\text{\rm a-LW}(X,F)\not\Longrightarrow S'\in\text{\rm a-MW}(F',X')$ 
(see \cite[Rem.2.1]{BLM1}).
Since w-convergent sequences in $Y'$ are bounded, $\text{\rm MW}(E,Y)\subseteq\text{\rm a-MW}(E,Y)$.
The last inclusion may be proper: e.g. the identity operator $I_{c_0}:c_0\to c_0$ is \text{\rm a-MW} 
yet not w-compact, and therefore $I_{c_0}$ is not \text{\rm MW}. In general, 
$\text{\rm K}(E,F)\not\subseteq\text{\rm a-LW}(E,F)\cup\text{\rm a-MW}(E,F)$, as the next example shows. 

\begin{example}\label{compact not a-LW}
{\em 
Define an operator $T\in\text{\rm K}(\ell^1,c)$ by 
$T\alpha=(\sum\limits_{k=1}^{\infty}\alpha_k)\cdot{\mathbb 1}_{\mathbb N}$.
As $\{Te_1\}=\{{\mathbb 1}_{\mathbb N}\}$ is not an \text{\rm LW}-subset of $c$,
$T\not\in\text{\rm a-LW}(\ell^1,c)$.
Furthermore, $T\not\in\text{\rm a-MW}(\ell^1,c)$.
Indeed, for the disjoint bounded sequence $(e_n)$ of unit vectors in $B_{\ell^1}$ 
and for any $f\in c'$ with $f({\mathbb 1}_{\mathbb N})=1$, 
the constant sequence $(f_n)$, $f_n\equiv f$, is w-convergent, yet 
$f_n(Te_n)\equiv 1\not\to 0$.

Since $f_n(Tx_n)\to 0$ for every disjoint bounded $(x_n)$ in $\ell^1$ 
and every w-null $(f_n)$ in $c'$, a w-convergent sequence $(f_n)$ in 
Definition~\ref{Main MWC operators}~a) can not be replaced by w-null $(f_n)$.}
\end{example}
\noindent
The implication below is straightforward. 
\begin{equation}\label{4a}
   T\in\text{\rm a-MW}(E,X) \ \& \  S\in\text{\rm L}(X,Y) \  \Longrightarrow \ 
   S  T\in\text{\rm a-MW}(E,Y).
\end{equation}

\subsection{Algebraic properties of \text{\rm a-MW}-operators.}
\begin{lemma}\label{Composition of MW}
For a pair of operators $T$ and $S$$:$ 
\begin{enumerate}[\em i)]
\item if $T\in\text{\rm a-MW}(E,F)$ and $S\in\text{\rm MW}(F,Y)$ then $S  T\in\text{\rm MW}(E,Y)$$;$
\item if $T\in\text{\rm a-MW}(E,X)$ and $S\in\text{\rm L}(X,Y)$ then $S  T\in\text{\rm a-MW}(E,Y)$.
\end{enumerate}
\end{lemma}
\begin{proof}
i) Let $T\in\text{\rm a-MW}(E,F)$ and $S\in\text{\rm MW}(F,Y)$. Then $T'\in\text{\rm a-LW}(F',E')$ by
Assertion~\ref{Bouras--Lhaimer--Moussa}(i), and $S'\in\text{\rm LW}(Y',F')$ by
Assertion~\ref{Meyer-Nieberg}(i). It follows from Lemma~\ref{Composition of LW}(i), that
$(S  T)'=T'  S'\in\text{\rm LW}(Y',E')$. Hence $S  T\in\text{\rm MW}(E,Y)$  by Assertion~\ref{Meyer-Nieberg}(i).\\
\noindent
ii) Follows from \eqref{4a}. 
\end{proof}

\begin{proposition}\label{MW-algebra}
For any Banach lattice $E$ the following holds.
\begin{enumerate}[]
\item $\text{\rm MW}(E)$ is a closed subalgebra of $\text{\rm W}(E)$ and is unital 
iff the identity operator $I_E$ is \text{\rm MW}.
\item $\text{\rm a-MW}(E)$ is a closed left ideal in $\text{\rm L}(E)$ 
$($and hence a subalgebra of $\text{\rm L}(E)$$)$, and is unital iff 
$I_E$ is \text{\rm a-MW}.
\end{enumerate}
\end{proposition}
\begin{proof}
Both $\text{\rm MW}(E)$ and $\text{\rm a-MW}(E)$ are closed subspaces of $\text{\rm W}(E)$ by 
Proposition \ref{LW-algebra} and Assertion~\ref{Meyer-Nieberg}. 
The closeness of $\text{\rm MW}(E)$ under the composition follows from Lemma \ref{Composition of MW}(i)(ii).
The fact that $\text{\rm a-MW}(E)$ is a right ideal in $\text{\rm L}(E)$ follows from Lemma \ref{Composition of MW}(ii).
Conditions on $I_E$ making algebras $\text{\rm MW}(E)$ and $\text{\rm a-MW}(E)$ to be unital, are trivial.
\end{proof}
\noindent
The following fact was established in \cite[Thm.2]{AkGo}.

\begin{assertion}\label{domination for aMW}
For all Banach lattices $E$ and $F$ the $\text{\rm a-MW}$-operators 
from $E$ to $F$ satisfy the domination property.
\end{assertion}
\noindent
In view of Assertion \ref{domination for aMW},
the next lemma follows from Proposition~\ref{prop elem}~ii).

\begin{lemma}\label{prop elem aMW}
Let an operator $T:E\to F$ possess the modulus. TFAE.
\begin{enumerate}[{\em i)}]
\item $T$ is a regularly $\text{\rm a-MW}$-operator.
\item $|T|$ is an $\text{\rm a-MW}$-operator.
\end{enumerate}
\end{lemma}

\begin{theorem}
The following statements hold. 
\begin{enumerate}[{\em i)}]
\item 
$\text{\rm r-a-MW}(E)$ is a subalgebra of $\text{\rm L}_r(E)$. Moreover, 
$$
   \text{\rm r-a-MW}(E)=\text{\rm L}_r(E) \ \Longleftrightarrow I_E\in\text{\rm a-MW}(E).
$$ 
\item 
If $E$ is Dedekind complete then 
$(\text{\rm r-a-MW}(E), \ \|\cdot\|_{\text{\rm r-a-MW}})$
is a closed Riesz subalgebra and an order ideal of $(\text{\rm L}_r(E), \ \|\cdot\|_r)$.
\end{enumerate}
\end{theorem}

\begin{proof}
i) It follows from Lemma~\ref{Composition of MW}(ii)
that $\text{\rm r-a-MW}(E)$ is a left algebra ideal and 
hence is a subalgebra of $\text{\rm L}_r(E)$. 
The condition on $I_E$ under 
which $\text{\rm r-a-MW}(E)=\text{\rm L}_r(E)$ is trivial.

ii) It follows from Lemma~\ref{prop elem aMW} 
that $\text{\rm r-a-MW}(E)$ is closed under
lattice operations, and hence is a Riesz subalgebra of $\text{\rm L}_r(E)$.

The rest of the proof is similar to the proof of Theorem \ref{a-LW-Riesz-algebra}.
\end{proof}

\subsection{}
In certain cases $\text{\rm a-LW}$- or $\text{\rm a-MW}$-operators may coincide with bounded 
operators.

\begin{proposition}\label{prop 13}
For any Banach space $X$ the following hold.
\begin{enumerate}[\em i)]
\item
If $F$ is an AL-space then $\text{\rm a-LW}(X,F)=\text{\rm L}(X,F)$.
\item
If $E$ is an AM-space then $\text{\rm a-MW}(E,Y)=\text{\rm L}(E,Y)$.
\end{enumerate}
\end{proposition}

\begin{proof}
i) Let $S:X\to F$ be a bounded operator and let $C$ be a relatively w-compact subset of $X$. 
Then $S(C)$ is relatively w-compact subset of $F$. By \cite[Thm.5.56]{AlBu}, 
$S(C)$ is an \text{\rm LW}-subset of $F$,
and hence $S\in\text{\rm a-LW}(X,F)$.

ii) Let $T\in\text{\rm L}(E,Y)$. Since $E'$ is an AL-space, it follows from i) that 
$T'\in\text{\rm a-LW}(Y',E')$. Then $T\in\text{\rm a-MW}(E,Y)$ by 
Assertion \ref{Bouras--Lhaimer--Moussa}(i).
\end{proof}

\subsection{}
The next result describes some restrictions on $E$ and $F$ under those
$\text{\rm a-MW}_+(E,F)\subseteq\text{\rm MW}(E,F)$. Recall that a Banach space $X$ has the 
{\em Dunford--Pettis property} whenever $f_n(x_n)\to 0$ for each \text{\rm w}-null
$(x_n)$ in $X$ and each \text{\rm w}-null $(f_n)$ in $X'$ (cf. \cite[p.341]{AlBu}). 
An operator $T:X\to Y$ is called a {\em weak Dunford--Pettis operator} whenever 
$f_n(Tx_n)\to f(x)$ for each $f_n\stackrel{\text{\rm w}}{\to}f$ in $Y'$ and each 
$x_n\stackrel{\text{\rm w}}{\to}x$ in $X$ \cite[p.349]{AlBu}. 

\begin{proposition}\label{necessary conditions for a-MW to be MW}
Let the norm in $E'$ be {\em o}-continuous, $F$ be Dedekind complete,
and $\text{\rm a-MW}_+(E,F)\subseteq\text{\rm MW}(E,F)$. Then either 
$E\in\text{\rm (PSP)}$ or $F$ has {\em o}-continuous norm.
\end{proposition}

\begin{proof}
Suppose $E\not\in\text{\rm (PSP)}$ and $F$ has no o-continuous norm. 
It suffices to construct a positive \text{\rm a-MW}-operator 
from $E$ to $F$ that is not \text{\rm MW}. 

By Assertion~\ref{Schur}, there exists a disjoint w-null sequence $(x_n)$ in $E_+$, that is not norm null. 
WLOG, $\|x_n\|=1$ for all $n$. There exists a sequence $(f_n)$ in $E'_+$ with $\|f_n\|=1$,  
such that 
\begin{equation}\label{1x}
   |f_n(x_n)|\ge\frac{1}{2} \ \ \ \ (\forall n\in{\mathbb N}).
\end{equation}
Define a positive operator $T: E\to\ell^\infty$ by 
\begin{equation}\label{2x}
   Tx:=[(f_n(x))_n] \ \ \ \ (\forall x\in E).
\end{equation}
Since $F$ is Dedekind complete 
and  the norm in $F$ is not order continuous, there exists a lattice embedding 
$S:\ell^\infty\to F$ (cf. \cite[Thm.4.51]{AlBu}). Then, for some $M>0$:
\begin{equation}\label{3x}
   \|Sa\|\ge M\|a\|_\infty \ \ \ (\forall a\in\ell^\infty).
\end{equation}
Since $\ell^\infty$ has the Dunford--Pettis property (cf. \cite[Thm.5.85]{AlBu}), $S$ is a weak Dunford--Pettis operator 
\cite[p.349]{AlBu}, that is,  for each $g_n\stackrel{\text{\rm w}}{\to}g$ in $F'$
and each \text{\rm w}-null $(y_n)$ in $\ell^\infty$:
\begin{equation}\label{4x}
   g_n(Sy_n)\to 0.
\end{equation} 
Let $(z_n)$ be disjoint bounded sequence in $E$ and let $g_n\stackrel{\text{\rm w}}{\to}g$ in $F'$.
Since $E'$ has o-continuous norm, $(z_n)$ is w-null in $E$ by Assertion \ref{E' is o-cont}. Then $(Tz_n)$ is w-null in $\ell^\infty$.
It follows from \eqref{3x} that $g_n(STz_n)\to 0$. Therefore, $ST\in\text{\rm a-MW}(E,F)$. 

By using of \eqref{1x}, \eqref{2x}, and \eqref{3x}, we obtain
\begin{equation}\label{5x}
   \|STx_n\|=\|S[(f_k(x_n))_k]\|\ge 
   M\|(f_k(x_n))_k\|_\infty\ge M|f_n(x_n)|\ge\frac{M}{2}>0.
\end{equation}
It follows from \eqref{5x} that $ST\not\in\text{\rm MW}(E,F)$, and we are done.
\end{proof}

The following result is an \text{\rm LW}-version of 
Proposition \ref{necessary conditions for a-MW to be MW}.

\begin{proposition}\label{sufficient conditions for a-LW subseteq LW}
If $\text{\rm a-LW}_+(E,F)\subseteq\text{\rm LW}(E,F)$, then either 
$E'$ has \text{\rm o}-co\-n\-ti\-nuous norm or $F'\in\text{\rm (PSP)}$.
\end{proposition}

\begin{proof}
Suppose, in contrary, that the norm in $E'$ is not \text{\rm o}-continuous and $F'\not\in\text{\rm (PSP)}$.
Then we can find a disjoint sequence $(x'_n)$ in $E'$ and $x'\in E'$ such that
$0\le x'_n \le x'$ and $\|x'_n\|=1$ for all $n\in{\mathbb N}$. By Assertion~\ref{Schur}(iii), 
there is a disjoint \text{\rm w}-null sequence $(f_n)$ in $F'_+$ that is not norm null. 
By passing to a subsequence and by scaling, we may assume that $\|f_n\|\ge 1$ for all 
$n\in{\mathbb N}$. Choose a sequence $(y_n)$ in $F_+$ with $\|y_n\|=1$ and  
$|f_n(y_n)|\ge\frac{1}{2}$ for all $n\in{\mathbb N}$. 
We define positive operators $T:E\to \ell^1$ and $S:\ell^1\to F$ by 
$$
   Tx:=(x'_k(x))_k \ \ \& \ \ \ Sa:=\sum_{k=1}^\infty a_k y_k .
$$ 
Let $(z_n)$ be a \text{\rm w}-null sequence in $E$ and $(f_n)$ 
be a disjoint sequence in $B_{F'}$. As $\ell^1$ has the Schur property, 
$\|Tz_n\|\to 0$, and hence
\begin{equation}\label{6x}
   |f_n(STz_n)|\le\|f_n\|\cdot\|S\|\cdot\|Tz_n\|\to 0.
\end{equation} 
By \cite[Prop.1]{EAS}, \eqref{6x} implies $ST\in\text{\rm a-LW}_+(E,F)$. 

If $ST$ were to be an \text{\rm LW}-operator, then $(ST)'$ 
would be an \text{\rm MW}-operator. However,
$$
   \|(ST)'(f_n)\|=\|T'S'(f_n)\|=\|\sum_{k=1}^\infty f_n(y_k)\|\ge |f_{n}(y_n)|\ge\frac{1}{2}
$$
for each $n$, which shows $(ST)'\notin\text{\rm MW}(F',E')$ and hence $ST\notin\text{\rm LW}(E,F)$.
This contradiction completes the proof.
\end{proof}

\begin{proposition}\label{about modulus T a T'}
Suppose the norm in $E'$ is \text{\rm o}-continuous, $F$ is Dedekind complete, 
and $F'$ has the Schur property. Then, for every order bounded $T: E\to F$, 
$|T|$ is \text{\rm a-MW} and $|T|'$ is \text{\rm a-LW}.
\end{proposition}

\begin{proof}
Let $(x_n)$ be a disjoint bounded sequence in $E$ and $f_n\stackrel{\rm w}{\to}f$ in $F'$. 
By the Schur property of $F'$, we have $\|f_n-f\|\to 0$. Hence
\begin{equation}\label{(**)}
  |(f_n-f)(|T|x_n)|\le\|f_n -f\|\cdot\||T|x_n\| \to 0.
\end{equation}
Since $E'$ has o-continuous norm, $(x_n)$ is w-null by Assertion~\ref{E' is o-cont} and hence
$f(|T|x_n)=|T|'(f)(x_n) \to 0$. It follows from \eqref{(**)} that $f_n(|T|x_n)\to 0$, 
and hence $|T|$ is \text{\rm a-MW}. 
By Assertion \ref{Bouras--Lhaimer--Moussa}~(i), $|T|'\in\text{\rm a-LW}(F',E')$.
\end{proof}

\subsection{Basic properties of \text{\rm l-MW}-ope\-rators.} 

\begin{proposition}\label{l-MW-domination}
Let $0\le S\le T\in\text{\rm l-MW}(E,F)$ then $S\in\text{\rm l-MW}(E,F)$.
\end{proposition}
\begin{proof}
Let $(x_n)$ be disjoint bounded sequence in $E$. Then $(|x_n|)$ is also disjoint bounded,
and then $T|x_n|\stackrel{\text{\rm w}}{\to}0$.
It follows from $|Sx_n|\le S|x_n|$ that
$$
   |f(Sx_n)|\le|f|(S|x_n|)\le|f|(T|x_n|)\to 0  \ \ \ \ (\forall f\in F'),
$$
and hence $Sx_n\stackrel{\text{\rm w}}{\to}0$. Thus, $S\in\text{\rm l-MW}(E,F)$.
\end{proof}

\begin{lemma}\label{l-MW-closed}
Let $\text{\rm l-MW}(E,Y)\ni T_n\stackrel{\|\cdot\|}{\to} T$. 
Then $T\in\text{\rm l-MW}(E,Y)$.
\end{lemma}
\begin{proof}
Let $(x_n)$ be a disjoint sequence in $B_E$, and $f\in Y'$. 
Take $\varepsilon>0$, and pick an $m\in\mathbb{N}$
with $\|T-T_m\|\le\varepsilon$. Since $T_m\in\text{\rm l-MW}(E,Y)$  
then $|f(T_mx_n)|\le\varepsilon$ for $n\ge n_0$. It follows from
$$
   |f(Tx_n)|\le|f(Tx_n)-f(T_mx_n)|+|f(T_mx_n)|\le 
$$
$$
   \|T'-T'_m\|\|f\|\|x_n\|+|f(T_mx_n)|\le(\|f\|+1)\varepsilon 
   \ \ \ \ (\forall n\ge n_0),
$$
that $f(Tx_n)\to 0$, and hence $T\in\text{\rm l-MW}(E,Y)$.
\end{proof}
\noindent
In view of Proposition \ref{l-MW-domination},
the next lemma follows from Proposition~\ref{prop elem}~ii).

\begin{lemma}\label{prop elem l-MW}
Let an operator $T:E\to F$ possess the modulus. TFAE.
\begin{enumerate}[{\em i)}]
\item $T\in\text{\rm r-l-MW}(E,F)$.
\item $|T|\in\text{\rm l-MW}(E,F)$.
\end{enumerate}
\end{lemma}

\begin{theorem}
The following statements hold. 
\begin{enumerate}[{\em i)}]
\item 
$\text{\rm r-l-MW}(E)$ is a subalgebra of $\text{\rm L}_r(E)$. Moreover, 
$$
   \text{\rm r-l-MW}(E)=\text{\rm L}_r(E) \ \Longleftrightarrow I_E\in\text{\rm l-MW}(E).
$$ 
\item 
If $E$ is Dedekind complete then 
$(\text{\rm r-l-MW}(E), \ \|\cdot\|_{\text{\rm r-l-MW}})$
is a closed Riesz subalgebra and an order ideal of $(\text{\rm L}_r(E), \ \|\cdot\|_r)$.
\end{enumerate}
\end{theorem}

\begin{proof}
i) Since each bounded operator is \text{\rm w}-continuous,
$\text{\rm r-l-MW}(E)$ is a left ideal and hence is 
a subalgebra of $\text{\rm L}_r(E)$. 
The condition on $I_E$ under 
that $\text{\rm r-l-MW}(E)=\text{\rm L}_r(E)$ is trivial.

ii) By Lemma~\ref{prop elem l-MW}, $\text{\rm r-l-MW}(E)$ is closed under
lattice operations, and hence is a Riesz subalgebra of $\text{\rm L}_r(E)$.
It follows from Assertion \ref{P-norm} and Lemma {l-MW-closed}
that $(\text{\rm r-l-LW}(E), \ \|\cdot\|_{\text{\rm r-l-LW}})$ is a Banach space.
The rest of the proof is similar to the proof of Theorem \ref{a-LW-Riesz-algebra}.
\end{proof}

\subsection{}
We include also the following semi-duality result. Recall that 
$X\in(\text{\rm GP})$ if each w$^\ast$-null sequence in $X'$ is w-null.

\begin{proposition}\label{l-BLM}
The following hold.
\begin{enumerate}[\em i)]
\item 
  $T'\in\text{\rm l-LW}(Y',E')\Longrightarrow T\in\text{\rm l-MW}(E,Y)$, and if
  additionally $E'\in(\text{\rm DDw$^\ast$P})$ then 
  $T'\in\text{\rm l-LW}(Y',E')\Longleftrightarrow T\in\text{\rm l-MW}(E,Y)$. 
\item 
  $S'\in\text{\rm l-MW}(F',X')\Longrightarrow S\in\text{\rm l-LW}(X,F)$, and if
  additionally $X\in(\text{\rm GP})$ then 
  $S'\in\text{\rm l-MW}(F',X')\Longleftrightarrow S\in\text{\rm l-LW}(X,F)$.
\end{enumerate}
\end{proposition}

\begin{proof}
i) 
Let  $T'\in\text{\rm l-LW}(Y',E')$. 
Take a disjoint bounded sequence $(x_n)$ in $E$. 
Then $(\hat{x}_n)$ is disjoint bounded in $E''$.
By Lemma \ref{l-LW-operators}, 
$T''\hat{x}_n\stackrel{\text{\rm w}^\ast}{\to}0$ in $Y''$, and hence 
$$
   y(Tx_n)=(T'y)x_n=\hat{x}_n(T'y)=T''\hat{x}_n(y)\to 0 \ \ \ \ (\forall y\in Y').
$$ 
Thus $Tx_n\stackrel{\text{\rm w}}{\to}0$, and hence $T\in\text{\rm l-MW}(E,Y)$.

Now, let $T\in\text{\rm l-MW}(E,Y)$ and $E'\in(\text{\rm DDw$^\ast$P})$.
Take a disjoint bounded sequence $(f_n)$ in $E''$. Since $E'\in(\text{\rm DDw$^\ast$P})$
then $f_n\stackrel{\text{\rm w}^\ast}{\to}0$ in $E''$, and hence 
$T''f_n\stackrel{\text{\rm w}^\ast}{\to}0$ in $Y''$.
Lemma \ref{l-LW-operators} implies $T'\in\text{\rm l-LW}(Y',E')$. 

ii) 
Let $S'\in\text{\rm l-MW}(F',X')$. Take a disjoint bounded sequence $(f_n)$ in $F'$.
Then $S'f_n\stackrel{\text{\rm w}}{\to}0$ in $X'$, and hence 
$S'f_n\stackrel{\text{\rm w}^\ast}{\to}0$ in $X'$. It follows from  
Lemma \ref{l-LW-operators} that $S\in\text{\rm l-LW}(X,F)$.

Now, let $S\in\text{\rm l-LW}(X,F)$ and $X\in(\text{\rm GP})$.
Take a disjoint bounded sequence $(f_n)$ in $F'$. By Lemma \ref{l-LW-operators},
$S'f_n\stackrel{\text{\rm w}^\ast}{\to}0$ in $X'$. Since $X\in(\text{\rm GP})$
then $S'f_n\stackrel{\text{\rm w}}{\to}0$, and hence $S'\in\text{\rm l-MW}(F',X')$.
\end{proof}

%%%%%%%%%%%%%%%%%%%%
\section{Order \text{\rm LW}- (\text{\rm MW}-) operators}
%%%%%%%%%%%%%%%%%%%%

Order \text{\rm LW}- (\text{\rm MW}-) operators were introduced recently in \cite{BLM3,BLM2}.
Here we continue to investigate these classes of operators.
The idea of definition of an order \text{\rm LW}-operator $T:E\to F$ 
lies in choosing in Assertion~\ref{Burkinshaw--Dodds},
$$
   \text{$A:=T[-x,x]$ for arbitrary $x\in E_+$; and $B:=B_{F'}$}. 
$$

\subsection{Main definitions.}
\begin{definition}\label{Main o-W operators}
{\em
A bounded operator $T:E\to F$ is called:
\begin{enumerate}[a)]
\item   {\em order} \text{\rm LW} (shortly, $T$ is \text{\rm o-LW}) 
if $T[0,x]$ is an L-weakly compact subset of $F$ for every $x\in E_+$ \cite[Def.2.1]{BLM2};
\item  {\em order} \text{\rm MW} (shortly, $T$ is \text{\rm o-MW}) 
if, for every disjoint bounded sequence $(x_n)$ in $E$ and every order bounded sequence $(f_n)$ in $F'$, 
we have $f_n(Tx_n)\to 0$ (see \cite[Def.2.2]{BLM3} and \cite[Def.2.2]{BLM2}). 
\end{enumerate}}
\end{definition}
\noindent
Clearly, 
$$
   \text{\rm W}(E,Y)\subseteq\text{\rm o-W}(E,Y);
$$
$$
   \text{\rm LW}(E,F)\subseteq\text{\rm o-LW}(E,F);
$$
$$
   \text{\rm MW}(E,F)\subseteq\text{\rm o-MW}(E,F).
$$
Since the operators $\text{I}_{c_0}$ and $\text{I}_{c}$ are not w-compact,
one can directly derive from Definition \ref{Main o-W operators} that
$$
   \text{I}_{c_0}\in\text{\rm o-LW}(c_0)\setminus\text{\rm LW}(c_0) \ \ \text{\rm and} \ \ 
   \text{I}_{c}\in\text{\rm o-MW}(c)\setminus\text{\rm MW}(c). 
$$
The following two assertions (see \cite[Thm.2.1 and Thm.2.2]{BLM2}) 
give important 
characterizations of \text{\rm o-LW}-operators.

\begin{assertion}\label{Bouras--Lhaimer--Moussa - order LW}
For an operator $T:E\to F$, TFAE.
\begin{enumerate}[\em (i)]
\item 
$T\in\text{\rm o-LW}(E,F)$.
\item 
$(|T'f_n|)$ is $\text{\rm w}^\ast$-null for every disjoint bounded $(f_n)$ in $F'$.
\item 
$f_n(Tx_n)\to 0$ for every order bounded $(x_n)$ in $E$ and every 
disjoint bounded $(f_n)$ in $F'$.
\end{enumerate}
\end{assertion}

\begin{assertion}\label{Bouras--Lhaimer--Moussa - order2}
$S\in\text{\rm o-LW}(E',F')$ iff 
$$
   (Sf_n)(y_n)=y_n(Sf_n)=(S'y_n)(f_n)\to 0
$$ 
for every order bounded $(f_n)$ in $E'$ and every 
disjoint bounded $(y_n)$ in $F$.
\end{assertion}
\noindent
The next semi-duality result \cite[Thm.2.3]{BLM2} plays a key 
role in the investigation of
$\text{\rm o-LW}$- and $\text{\rm o-MW}$-operators.

\begin{assertion}\label{Bouras--Lhaimer--Moussa - order}
The following hold.
\begin{enumerate}[\em (i)]
\item 
$T\in\text{\rm o-MW}(E,F)\Longleftrightarrow T'\in\text{\rm o-LW}(F',E')$.
\item 
$S'\in\text{\rm o-MW}(F',E')\Longrightarrow S\in\text{\rm o-LW}(E,F)$.
\end{enumerate}
\end{assertion}
\noindent
It was also observed in \cite[Rem.2.3]{BLM2} that, in general, 
$$
S\in\text{\rm o-LW}(E,F)\not\Longrightarrow S'\in\text{\rm o-MW}(F',E'),
$$ 
as  
$I_{c_0}\in\text{\rm o-LW}(c_0)$, 
but $(I_{c_0})'=I_{\ell^1}\not\in\text{\rm o-MW}(\ell^1)$.

\subsection{}
A rank one bounded operator need not to be \text{\rm o-LW}. 
Consider the operator $T:\ell^2\to\ell^\infty$ defined by
$Ta=(\sum\limits_{k=1}^\infty a_k){\mathbb 1}_{\mathbb N}$.
Then $Te_n={\mathbb 1}_{\mathbb N}$ and, as ${\mathbb 1}_{\mathbb N}\not\in(\ell^\infty)^a$, 
the operator $T$ is not \text{\rm o-LW}. The following result is an improvement 
of \cite[Thm.2.5]{BLM2}:

\begin{proposition}\label{regular operators are o-MW}
For a Banach lattice $E$, TFAE.
\begin{enumerate}[\em i)]
\item 
The norm in $E'$ is \text{\rm o}-continuous. 
\item 
For every $F$, ${\cal L}_{ob}(E,F)\subseteq\text{\rm o-MW}(E,F)$.
\item 
For every $F$, $\text{\rm L}_r(E,F)\subseteq\text{\rm o-MW}(E,F)$.
\item 
For every $F$, $\text{\rm K}(E,F)\subseteq\text{\rm o-MW}(E,F)$.
\item 
Each rank one bounded operator $T:E\to E$ is \text{\rm o-MW}.
\end{enumerate}
\end{proposition}

\begin{proof} 
i) $\Longleftrightarrow$ ii) is in \cite[Thm.2.5]{BLM2} 
and the implication ii) $\Longrightarrow$ iii) is trivial. 

iii) $\Longrightarrow$ i). If not, there is a disjoint bounded sequence $(x_n)$ in $E$ 
and a functional $f\in E'$ such that $f(x_n)\ge 1$ for all $n\in\mathbb{N}$
by Assertion \ref{E' is o-cont}. The identity operator $T=I: E\to E$ is \text{\rm o-MW} by $iii)$.
However, for the order bounded sequence $(f_n)$ in $F'$ with 
$f_n\equiv f$, it holds $f_n(Tx_n)=f(x_n)\ge 1$, and hence $f_n(Tx_n)\not\to 0$, 
violating Assertion~\ref{Bouras--Lhaimer--Moussa - order2}. The obtained 
contradiction proves the implication.

i) $\Longrightarrow$ iv).
Let $T\in\text{\rm K}(E,F)$ and let $x\in F'_+$. Since $T'\in\text{\rm K}(F',E')$, then
$T'[0,x]$ is relatively compact and hence is almost order bounded. As the norm in $E'$
is \text{\rm o}-continuous, $T'[0,x]$ is an \text{\rm LW}-subset of $E'$ 
by Assertion~\ref{Meyer 3.6.2}.
This yields $T'\in\text{\rm o-LW}(F',E')$, and hence $T\in\text{\rm o-MW}(E,F)$,
by Assertion \ref{Bouras--Lhaimer--Moussa - order}(i).

iv) $\Longrightarrow$ v). It is obvious.

v) $\Longrightarrow$ i).
If the norm in $E'$ is not \text{\rm o}-continuous, 
take $g\in E'_+\setminus(E')^a$. 
Let $T$ be a bounded rank one operator
defined by $T(x)=g(x)y_0$ with $y_0\in E_+\setminus\{0\}$.
Then $T\in\text{\rm o-MW}(E)$. Choose $f\in E'$ with $f(y_0)=1$.
Since $T'f=f(y_0)g=g\not\in(E')^a$, $T'[0,f]$ is not an \text{\rm LW}-subset of $E'$ 
as each \text{\rm LW}-subset of $E'$ must lie in $(E')^a$. 
Therefore $T'\not\in\text{\rm o-LW}(E')$. This yields $T\not\in\text{\rm o-MW}(E)$ 
by Assertion \ref{Bouras--Lhaimer--Moussa - order}(i). 
The obtained contradiction completes the proof.
\end{proof}

\begin{corollary}\label{regular oMwc which are not regularly oMwc}
There is a Banach lattice $E$ with 
$\text{\rm L}_r(E)\subsetneqq\text{\rm o-MW}(E)$. 
\end{corollary}

\begin{proof}
Let $E$ be as in Example \ref{Krengel}. It follows from Proposition \ref{regular operators are o-MW} 
that $\text{\rm L}_r(E)\subseteq\text{\rm o-MW}(E)$.
Consider the operator $T$ defined by (\ref{1kr}) in Example \ref{Krengel} for 
$\alpha_n=2^{-\frac{n}{3}}$. As $T\in\text{\rm K}(E)$, applying Proposition 
\ref{regular operators are o-MW} again gives $T\in\text{\rm o-MW}(E)$.
However $|T|$ does not exist and hence $T\not\in\text{\rm L}_r(E)$.
\end{proof}

\begin{corollary}\label{rank one op1}
Let $E$ and $F$ be Banach lattices with $\dim(E)\ge 1$ and $\dim(F)\ge 1$ then$:$ 
\begin{enumerate}[\em i)]
\item $E'$ has $\text{\rm o}$-continuous norm $\Longleftrightarrow$ 
every rank one bounded operator $T:E\to F$ is \text{\rm o-MW}$;$
\item $F$ has $\text{\rm o}$-continuous norm $\Longleftrightarrow$ 
every rank one bounded operator $T:E\to F$ is \text{\rm o-LW}.
\end{enumerate}
\end{corollary}

\begin{proof}
i) It is a part of Proposition \ref{regular operators are o-MW}.

ii) The implication ($\Longrightarrow$) is obvious.

ii) ($\Longleftarrow$) If not, there exist $y\in F_+$ and a disjoint sequence $(y_n)$ in $[0,y]$,
that is not norm null. Choose $x\in E_+$, $f\in E'_+$ with $f(x)=1$ and define 
a rank one operator $T$ as $Tz=f(z)y$. 
Then $Tx=y$ and $(y_n)$ is a disjoint sequence in $\text{sol}\,(T[0,x])$, 
and it is not norm null, violating that $T$ 
must be \text{\rm o-LW}. A contradiction.
\end{proof}

\subsection{}
It follows directly from the definitions:   
$$
   \text{\rm LW}(E,F)\subseteq\text{\rm o-LW}(E,F)\cap\text{\rm a-LW}(E,F);
$$
$$
   \text{\rm MW}(E,F)\subseteq\text{o-MW}(E,F)\cap\text{a-MW}(E,F).
$$
If $F$ has $\text{\rm o}$-continuous norm, then 
${\cal L}_{ob}(E,F)\subseteq\text{\rm o-LW}(E,F)$, and if
$E'$ has o-continuous norm, 
then ${\cal L}_{ob}(E,F)\subseteq\text{\rm o-MW}(E,F)$ 
by \cite[Thm.2.5]{BLM2}. Clearly,
$$
   \text{I}_{c_0}\in\text{\rm o-LW}(c_0)\setminus\text{\rm a-LW}(c_0)\ \text{\rm and}\ 
   \text{I}_{\ell^1}\in\text{\rm a-LW}(\ell^1)\setminus\text{\rm a-MW}(\ell^1).
$$
It can be easily shown that 
$\text{I}_{\ell^2}\in\text{\rm o-MW}(\ell^2)\setminus\text{\rm a-MW}(\ell^2)$
(cf. Remark~\ref{reflexive not aMW}). 
By Proposition \ref{prop 13}, for each $X$: 
$\text{\rm o-LW}(X,F)=\text{\rm a-LW}(X,F)=\text{\rm L}(X,F)$ 
if $F$ is an AL-space; 
and $\text{\rm o-MW}(E,X)=\text{\rm a-MW}(E,X)=\text{\rm L}(E,X)$ 
if $E$ is an AM-space.
In general, $\text{\rm a-MW}(E,Y)\not\subseteq\text{\rm o-MW}(E,Y)$. 
Indeed,
\begin{equation}\label{(gav5)}
\text{I}_{C[0,1]}\in\text{\rm a-MW}(C[0,1])\ 
\text{\rm (e.g., by Proposition~\ref{prop 13}~ii))}.
\end{equation}
However, 
\begin{equation}\label{(gav6)}
\text{I}_{C[0,1]}\not\in\text{\rm o-MW}(C[0,1]) \ 
\text{\rm (e.g., by Proposition~\ref{T is o-MW} below)}.
\end{equation}
It follows directly that $\text{\rm a-LW}(\ell^2)\subseteq\text{\rm o-LW}(\ell^2)$.
Then $\text{\rm a-MW}(\ell^2)\subseteq\text{\rm o-MW}(\ell^2)$ by the duality.
Since $\text{I}_{\ell^2}\in\text{\rm o-MW}(\ell^2)\setminus\text{\rm a-MW}(\ell^2)$ then $\text{\rm a-MW}(\ell^2)\subsetneqq\text{\rm o-MW}(\ell^2)$.

\subsection{Algebraic properties of  \text{\rm o-LW}- and 
\text{\rm o-MW}- operators.}
The first of the following two elementary lemmas is 
a version of Lemma \ref{Composition of LW} 
for \text{\rm o-LW}-operators.

\begin{lemma}\label{Composition of oLW}
If $S\in{\cal L}_{ob}(E,F)$ and $T\in\text{\rm o-LW}(F,G)$ 
then $T  S\in\text{\rm o-LW}(E,G)$.
\end{lemma}
\noindent
Since \text{\rm o-LW}-operators satisfy the 
domination property by \cite[Cor.2.3]{BLM2}, 
the second lemma follows from
Proposition~\ref{prop elem}~ii).

\begin{lemma}\label{prop elem oLW}
Let an operator $T\in{\cal L}(E,F)$ possess the modulus. TFAE.
\begin{enumerate}[{\em i)}]
\item $T$ is a regularly $\text{\rm o-LW}$-operator.
\item $|T|$ is an $\text{\rm o-LW}$-operator.
\end{enumerate}
\end{lemma}
\noindent
We include a version of Theorem \ref{a-LW-Riesz-algebra} 
for \text{\rm o-LW}-operators.

\begin{theorem}\label{o-LW-algebra}
The following statements hold. 
\begin{enumerate}[{\em i)}]
\item 
$\text{\rm r-o-LW}(E)$ is a subalgebra of $\text{\rm L}_r(E)$. Moreover, 
$$
   \text{\rm r-o-LW}(E)=\text{\rm L}_r(E) \ \Longleftrightarrow I_E\in\text{\rm o-LW}(E).
$$ 
\item 
If $E$ is Dedekind complete then 
$(\text{\rm r-o-LW}(E), \ \|\cdot\|_{\text{\rm r-o-LW}})$
is a closed Riesz subalgebra and an order ideal of $(\text{\rm L}_r(E), \ \|\cdot\|_r)$.
\end{enumerate}
\end{theorem}
\begin{proof}
i) Due to Lemma \ref{Composition of oLW},  
$\text{\rm r-o-LW}(E)$ is a right algebra ideal and hence 
is a subalgebra of $\text{\rm L}_r(E)$. 
The condition on $I_E$, that makes $\text{\rm r-o-LW}(E)=\text{\rm L}_r(E)$,
is clear.

ii) It follows from Lemma \ref{prop elem oLW} that the subalgebra 
$\text{\rm r-o-LW}(E)$ is closed under
modulus, and hence it is Riesz subalgebra of $\text{\rm L}_r(E)$.

As $\text{\rm o-LW}(E)$ is a closed subspace of $\text{\rm L}(E)$ by \cite[Prop.2.1]{BLM2},
the rest of the proof is similar to the proof of Theorem \ref{a-LW-Riesz-algebra}.
\end{proof}

The following lemma is a version of Lemma~\ref{Composition of MW} 
for \text{\rm o-MW}-operators.

\begin{lemma}\label{Composition of oMW}
If $T\in\text{\rm o-MW}(E,F)$ and $S\in{\cal L}_{ob}(F,G)$, 
then $S  T\in\text{\rm o-MW}(E,G)$.
\end{lemma}

\begin{proof}
Let $T\in\text{\rm o-MW}(E,F)$ and $S\in{\cal L}_{ob}(F,G)$. 
Then $T'\in\text{\rm o-LW}(F',E')$ by
Assertion~\ref{Bouras--Lhaimer--Moussa - order}~(i). It follows from 
Lemma \ref{Composition of oLW} 
that $(S  T)'=T'  S'\in\text{\rm o-LW}(G',E')$.
Applying Assertion~\ref{Bouras--Lhaimer--Moussa - order}~(i) again, 
we get
$S  T\in\text{\rm o-MW}(E,G)$.
\end{proof}
\noindent

It is straightforward that \text{\rm o-MW}-operators 
satisfy the domination property.
Thus, the next lemma follows from Proposition~\ref{prop elem}~ii).

\begin{lemma}\label{prop elem oMW}
Let an operator $T\in{\cal L}(E,F)$ possess the modulus. TFAE.
\begin{enumerate}[{\em i)}]
\item $T$ is a regularly $\text{\rm o-MW}$-operator.
\item $|T|$ is an $\text{\rm o-MW}$-operator.
\end{enumerate}
\end{lemma}

The next result is a version of Proposition \ref{MW-algebra} 
for \text{\rm o-MW}-operators. 
We omit its proof as it is similar to the proof of Theorem \ref{o-LW-algebra}, 
with replacement of Lemma \ref{prop elem oLW} by Lemma \ref{prop elem oMW}.

\begin{theorem}\label{o-MW-algebra}
The following statements hold. 
\begin{enumerate}[{\em i)}]
\item 
$\text{\rm r-o-MW}(E)$ is a subalgebra of $\text{\rm L}_r(E)$. Moreover, 
$$
   \text{\rm r-o-MW}(E)=\text{\rm L}_r(E) \ \Longleftrightarrow I_E\in\text{\rm o-MW}(E).
$$ 
\item 
If $E$ is Dedekind complete then 
$(\text{\rm r-o-MW}(E), \ \|\cdot\|_{\text{\rm r-o-MW}})$
is a closed Riesz subalgebra and an order ideal of $(\text{\rm L}_r(E), \ \|\cdot\|_r)$.
\end{enumerate}
\end{theorem}
\noindent
So, in certain cases, $\text{\rm r-o-LW}(E)$ or $\text{\rm r-o-MW}(E)$ 
may coincide with $\text{\rm L}_r(E)$. The following proposition develops this observation further.

\begin{proposition}\label{prop 14}
The following statements hold.  
\begin{enumerate}[\em i)]
\item
If $F$ is an AL-space then $\text{\rm r-o-LW}(E,F)=\text{\rm L}_r(E,F)$.
\item
If $E$ is an AM-space then $\text{\rm r-o-MW}(E,F)=\text{\rm L}_r(E,F)$.
\end{enumerate}
\end{proposition}

\begin{proof}
i) Let $S=S_1-S_2$ with both $S_1,S_2:E\to F$ positive, and let $x\in X_+$. 
Then $S_1([-x,x])$ and $S_2([-x,x])$ are both relatively w-compact in $F$.
By \cite[Thm.5.56]{AlBu}, the sets $S_1([-x,x])$ and $S_2([-x,x])$ are both
\text{\rm LW}-subsets of $F$. Then $S_1$ and $S_2$ are 
positive \text{\rm o-LW}-operators and 
hence $S$ is a regularly \text{\rm o-LW}-operator. 

ii) Let $T\in\text{\rm L}_r(E,F)$. Since $E'$ is an AL-space, 
it follows from i) that 
$T'\in\text{\rm r-o-LW}(F',E')$. Then $T\in\text{\rm r-o-MW}(E,F)$ by 
Assertion~\ref{Bouras--Lhaimer--Moussa - order}~(i).
\end{proof}

\subsection{Miscellanea.}
An operator $T:E\to X$ is said to be AM-{\em compact}, 
if $T[0,x]$ is relatively compact in $X$ for each $x\in E_+$.  
An operator $T:E\to X$ is said to be b-AM-{\em compact}, 
if it carries each b-order bounded subset of $E$
into a relatively compact subset of $X$ \cite{AM}.
Denote by $\text{\rm AM}(E,X)$ (by $\text{\rm b-AM}(E,X)$) 
the set of all AM-compact (resp. b-AM-co\-m\-pact)
operators from $E$ to $X$. 
Clearly, $\text{\rm b-AM}(E,X)\subseteq \text{\rm AM}(E,X)$.

\begin{theorem}\label{every o-LW is AMc}
Let $F^a$ be discrete. 
Then $\text{\rm o-LW}(E,F)\subseteq\text{\rm AM}(E,F)$.
\end{theorem}

\begin{proof}
Let $T\in\text{\rm o-LW}(E,F)$. Take $x\in E_+$. 
Since $T[0,x]$ is an \text{\rm LW}-subset of $F$,
for every $\varepsilon>0$, there is $u_\varepsilon\in F_+^a$ with
\begin{equation}\label{(*)}
   T[0,x]\subseteq [-u_\varepsilon,u_\varepsilon]+\varepsilon B_F,
\end{equation}
by Assertion~\ref{Meyer 3.6.2}. 
Let $\varepsilon>0$.  Since $F^a$ is discrete, 
$[-u_\varepsilon,u_\varepsilon]$ is compact. 
As $\varepsilon>0$ is arbitrary, 
\eqref{(*)} implies that $T[0,x]$ is totally bounded and 
hence is relatively compact.
Since $x\in E_+$ is arbitrary, $T\in \text{\rm AM}(E,F)$.
\end{proof}
\noindent
While \text{\rm LW}-operators are \text{\rm w}-compact, 
\text{\rm o-LW}-operators are generally not, e.g. 
$\text{I}_{c_0}\in\text{\rm o-LW}(c_0)\setminus\text{W}(c_0)$.
The following corollary describes some conditions for 
w-compactness of \text{\rm o-LW}-operators. 

\begin{corollary}\label{when o-LW is W}
Let $F^a$ be discrete and let either $E'$ be $KB$ or $F$ be reflexive. 
Then $\text{\rm o-LW}(E,F) \subseteq \text{\rm W}(E,F)$.
\end{corollary}

\begin{proof}
Let $T\in\text{\rm o-LW}(E,F)$. 
By Theorem~\ref{every o-LW is AMc}, $T\in\text{\rm AM}(E,F)$.
In view of \cite[Thm.2.4]{AMA}, $T\in\text{\rm W}(E,F)$.
\end{proof}

\begin{proposition}\label{T is o-MW}
Let $T\in\text{\rm L}_r(E,F)$ and  $x\in E_+$. 
Then the restriction of $T:(I_x,\|.\|_x)\to F$ is \text{\rm o-MW}. 
\end{proposition}

\begin{proof}
We may suppose $T\ge 0$. 
The principal order ideal $I_x$ of $E$ is a Banach lattice equipped 
with the norm $\|y\|_x=\inf\{\lambda: |y|\le\lambda x\}$.
Let $(x_n)$ be a disjoint order bounded sequence in $I_x$. 
Then $x_n \to 0$ in $\sigma(I_x,I'_x)$. 
Since the lattice operations are sequentially 
w-continuous (cf. \cite[Prop.13]{Wnuk3}), 
we have $|x_n| \to 0$ in $\sigma(I_x,I'_x)$. 
Since $T$ is positive, $T: (I_x,\|.\|_x)\to F$ is continuous, and hence 
$T$ is also continuous 
when the both spaces are equipped with their w-topologies. 
Let  a sequence $(f_n)$ be order bounded in $F'$, 
say $|f_n|\le f\in F'$. Then 
$$
   |f_n(Tx_n)|\le |f_n|(|Tx_n|)\le |f_n|(T|x_n|)\le f(T|x_n|),
$$
and $|f_n(Tx_n)|$ can be made arbitrary small, if $n$ is large enough. 
Therefore $T\in \text{\rm o-MW}(I_x,F)$.
\end{proof}

\begin{definition}\label{semicompact operator}
{\em An operator $T:X\to F$ is called {\em semicompact} 
if it takes bounded subsets of $X$ onto 
almost order bounded subsets of $F$.}
\end{definition}
\noindent
It is straightforward that $\text{\rm K}(X,F)\subseteq\text{\rm semi-K}(X,F)$, 
where
$\text{\rm semi-K}(X,F)$ denotes the space of all semicompact 
operators from $X$ to $F$.
It is also well known that 
$\text{\rm semi-K}(X,F)\subseteq\text{\rm W}(X,F)$ for every $X$,
whenever the norm in $F$ is o-continuous.
In view of Assertion \ref{Schur}(v), if $F\in(\text{\rm PSP})$ then
$\text{\rm semi-K}(X,F)=\text{\rm W}(X,F)$ for every $X$.

\begin{proposition}\label{TS is LW}
Let $S\in\text{\rm semi-K}(E,F)$ and $T\in\text{\rm o-LW}(F,G)$. 
Then  $T  S\in\text{\rm LW}(E,G)$. 
\end{proposition}

\begin{proof}
WLOG $\|T\|\leq 1$. 
Let $\varepsilon>0$ be arbitrary. Pick $u_\varepsilon\in F_+$ such that
$S(B_E)\subseteq [-u_\varepsilon,u_\varepsilon]+\varepsilon B_F$. Then
$$
   TS(B_E)\subseteq T[-u_\varepsilon,u_\varepsilon]+\varepsilon T(B_F)
                          \subseteq T[-u_\varepsilon,u_\varepsilon]+\varepsilon B_G .
$$
Since $T\in \text{\rm o-LW}(F,G)$, the set $T[-u_\varepsilon,u_\varepsilon]$ 
is an \text{\rm LW}-subset of $G$. Since $\varepsilon>0$ is arbitrary,  $TS(B_E)$ 
is an \text{\rm LW}-subset of  $G$.
\end{proof}

\noindent
Finally, we include the following result. 
By  \cite{MAE}, a bounded operator $T:E\to Y$ 
is {\em almost order weakly compact} (shortly, $T\in\text{\rm ao-W}(E,Y)$), 
whenever $T$ maps almost order bounded subsets of $E$ 
onto relatively weakly compact subsets of $Y$. 

\begin{proposition}
$\text{\rm ao-W}(E,Y)=\text{\rm o-W}(E,Y)$.
\end{proposition}

\begin{proof}
Clearly, $\text{\rm ao-W}(E,Y)\subseteq\text{\rm o-W}(E,Y)$.

Let $T\in\text{\rm o-W}(E,Y)$, and let $A$ be 
an almost order bounded subset of $E$.
Take $\varepsilon>0$. There exists $x\in E_+$ 
with $A\subseteq [-x,x]+\varepsilon B_E$. 
Then 
$$
   T(A)\subseteq T[-x,x]+\varepsilon\|T\|B_Y \subseteq 
   \text{cl}_{\|\cdot\|}(T[-x,x])+\varepsilon\|T\|B_Y,
$$ 
with a \text{\rm w}-compact subset $\text{cl}_{\|\cdot\|}(T[-x,x])$ of $Y$. 
Since $\varepsilon>0$ is arbitrary, \cite[Thm.3.44]{AlBu} 
implies that $T(A)$ is relatively \text{\rm w}-compact, 
and hence  $T\in\text{\rm ao-W}(E,Y)$.
\end{proof}

{\tiny 
%%%%%%%%%%%%%%%%%%%%
}
%\newpage

\end{document}